\documentclass[reqno]{amsart}
\usepackage{amsfonts, amsmath, amsthm, amssymb, latexsym, xfrac, mathrsfs, braket}

\newcommand{\C}{\mathbb{C}}
\newcommand{\R}{\mathbb{R}}

\newcommand{\N}{\mathbb{N}}

\newcommand{\m}{\cdot} 
\newcommand{\tm}{\subseteq} 

\newcommand{\nrm}[1]{\left\lVert#1\right\rVert}
\newcommand{\abs}[1]{\ensuremath{\left\vert#1\right\vert}}

\newcommand{\te}{\textrm}

\makeatletter
\@namedef{subjclassname@2020}{\textup{2020} Mathematics Subject Classification}
\makeatother

\theoremstyle{plain}
\newtheorem{theorem}{Theorem}[section]
\newtheorem{corollary}[theorem]{Corollary}
\newtheorem{lemma}[theorem]{Lemma}
\newtheorem{proposition}[theorem]{Proposition}
\theoremstyle{definition}
\newtheorem{definition}[theorem]{Definition}
\theoremstyle{remark}
\newtheorem{remark}[theorem]{Remark}

\numberwithin{equation}{section}

\title[Dunkl-Laplace transform and Macdonald's hypergeometric series]
{The Dunkl-Laplace transform and Macdonald's hypergeometric series}
\author{Dominik Brennecken and Margit R\"osler} 
\address{Institut f\"ur Mathematik, Universit\"at Paderborn, Warburger Str. 100, D-33098 Paderborn, Germany}
\email{bdominik@math.upb.de, roesler@math.upb.de}


\thanks{The authors were supported by DFG grant RO 1264/4-1.}

\begin{document}
\date{\today}

\begin{abstract}
We continue a program  generalizing classical results from the analysis on symmetric cones to the Dunkl setting for root systems of type $A$.
In particular, we prove a Dunkl-Laplace transform identity for Heckman-Opdam hypergeometric functions of type $A$ and more generally, for the associated Opdam-Cherednik kernel. This is achieved by analytic continuation from a Laplace transform identity for non-symmetric  Jack polynomials which was stated, for the symmetric case, as a key conjecture by Macdonald in \cite{Mac13}. Our proof for the Jack polynomials is based on Dunkl operator techniques and the raising operator of Knop and Sahi. 
Moreover, we use these results to establish Laplace transform identities between hypergeometric series in terms of Jack polynomials.  Finally, we conclude with a Post-Widder inversion formula for the Dunkl-Laplace transform.

\end{abstract}

\maketitle

\section{Introduction}

The Laplace transform is an important tool in various areas of harmonic analysis and forms a cornerstone in the analysis  on symmetric cones, see \cite{FK94}. In particular, there are important Laplace transform identities between $\,_pF_q$-hypergeometric functions on a symmetric cone, which are given as expansions with respect to the associated spherical polynomials, c.f. \cite[Chap.XV]{FK94}. For  cones of positive definite matrices, such   hypergeometric series trace back to ideas of Bochner and were studied in detail by Herz \cite{Her55}, where they were actually defined recursively by means of the Laplace transform. For important further developments see for instance \cite{Con63, GR89, Kan93}.  Multivariable hypergeometric series  
have found many applications in multivariate statistics \cite{Mui82}, but also in number theory and mathematical physics.

In his unpublished manuscript \cite{Mac13} from the 1980ies,  Macdonald introduced hypergeometric series in terms of Jack polynomials, which include the 
hypergeometric functions on symmetric cones  as special cases. 
 In this context, he also introduced a generalization of the Laplace transform for radial functions on symmetric cones, but many statements in \cite{Mac13} 
remained at a formal level. 
 
Radial analysis  on symmetric cones 
is closely related to Dunkl theory for root systems of type $A$, and also Macdonald's concepts have a natural interpretation within Dunkl theory,
because the $_0F_0$-hypergeometric function, which replaces the exponential kernel in the Macdonald's Laplace transform, is just a Dunkl-Bessel function of type $A$.
The connection of the concepts in \cite{Mac13} to Dunkl theory was already observed by Baker and Forrester in their seminal papers \cite{BF97, BF98} related to the study of Calogero-Moser-Sutherland models.

To become more precise on these connections, 
consider a symmetric cone $\Omega =G/K$ inside a simple Euclidean Jordan algebra $V$ of rank $n$ and with Peirce constant $d,$ which takes only specific integer values. Let 
$F\in L^1_{loc}(\Omega)$ be $K$-invariant, that is of the form $F(x)= f(\text{spec}(x)),$ where $\text{spec}(x) \in \mathbb R_+^n$ with $\mathbb R_+ = ]0,\infty[$ denotes the set of eigenvalues of $x$ ordered by size. Then for $y\in \Omega$, the Laplace transform 
 $$ LF(y) = \int_\Omega e^{-(x\vert y)}F(x)dx $$
 depends only on $\eta= \text{spec}(y)\in \mathbb R_+^n$ and can be written as 
 \begin{equation}\label{Bessel_Laplace_cone} LF(y) = \text{const}\cdot \int_{\mathbb R_+^n}  J_{d/2}(-\xi, \eta) f(\xi) \omega_{d/2}(\xi)d\xi\end{equation}
 where for $k \geq 0$, 
 $$ \omega_k(\xi) = \prod_{i<j} |\xi_i-\xi_j|^{2k} $$
 and  
 $$ J_k(z,w) = \sum_{\lambda\in \Lambda_n^+}  \frac{1}{|\lambda|!} \frac{C_\lambda(z;\alpha)C_\lambda(w;\alpha)}{C_\lambda(\underline 1;\alpha)} =\, _0F_0^\alpha(z,w),\quad \alpha = \frac{1}{k} \in\, ]0, \infty]. $$
 Here $\Lambda_n^+$ denotes the set of partitions with at most $n$ parts and the $ C_\lambda(\,.\,;\alpha)$ are the (symmetric) Jack polynomials of index $\alpha$ in $C$-normalization as in Lemma \ref{C_normalization} below. See e.g. \cite{R20} for some details. On the other hand, it is well-known  
 (c.f. \cite{BF98} and Remark \ref{Dunkl_hyper_connect}) that for arbitrary $k \geq 0$, the function $J_k$ 
 coincides with a Dunkl-Bessel function of type $A_{n-1}$ with multiplicity $k.$ 
 
 \smallskip
  Macdonald \cite{Mac13} considered the  Laplace transform  \eqref{Bessel_Laplace_cone} with the Bessel function $e(z,w) =\,_0F_0^{\alpha}(z,w) $ for arbitrary indices  $\alpha >0$. Many of his calculations were of a formal nature and rested on the following "Conjecture (C)" about the Laplace transform of Jack polynomials: 
 For $k \geq 0,$ let
 $$\mu_0:=k(n-1)\, \text{ and }\, \Delta(z):= \prod_{j=1}^n z_j \text{ for }\, z\in \mathbb C^n$$
 and write $C_\lambda (z) := C_\lambda(z; \tfrac{1}{k})$ for abbreviation.  Then for  all $\lambda\in \Lambda_n^+, \,y\in \mathbb R_+^n$ and all $\mu\in \mathbb C$ with $\text{Re} \,\mu >\mu_0\,,$   
  \begin{equation}\label{Mac_C} \int_{\mathbb R_+^n} J_k(-y,x)\,C_\lambda(x) \Delta(x)^{\mu-\mu_0-1}\omega_k(x)dx\,=\, \Gamma_n(\lambda +\underline \mu)\, C_\lambda\bigl(\tfrac{1}{y}\bigr)\,\Delta(y)^{-\mu}.\end{equation}
Here $\Gamma_n(\lambda)= \Gamma_n(\lambda;k)$ is Macdonald's multivariate gamma function defined in formula 
 \eqref{Gamma_n} below. 
 In \cite{R20}, we gave a rigorous treatment of the Dunkl-Laplace transform
  $$ \mathcal Lf(z) = \int_{\mathbb R_+^n} E_k(-z,x) f(x) \omega_k(x)dx$$
  where compared to Macdonald's version, the Bessel function is replaced by the Dunkl kernel $E_k$ of type $A_{n-1}$ and multiplicity $k.$  This transform was already considered by  Baker and Forrester \cite{BF98} and 
 later used in \cite{SZ07}, but convergence issues had remained open for a long time. 
  Formula \eqref{Mac_C} generalizes a Laplace transform identity for spherical polynomials on a symmetric cone, which 
 is in turn a consequence of the following important 
Laplace transform identity for the generalized power functions $\Delta_{s}$, $s = (s_1, \ldots, s_n) \in \mathbb C^n$ with $\text{Re}\,s_j > \frac{d}{2}(j-1)$
(see \cite[Chapter VII]{FK94}):
\begin{equation}\label{powercone}
\int_{\Omega} e^{-(y\vert x)} \Delta_{s}(x) \Delta(x)^{-m/n} dx\,=\, \Gamma_\Omega(s) \Delta_{s}(y^{-1}) \quad (y\in \Omega), 
\end{equation}
where $m$ is the dimension of the Jordan algebra, $\Delta$ denotes the Jordan  determinant and $\Gamma_{\Omega}$ is the Gindikin gamma function associated with $\Omega$.
 Taking $K$-means in \eqref{powercone}, one gets the same Laplace transform  identity  for the spherical functions of $\Omega$, which may be parametrized as
$$\varphi_\lambda(x) = \int_K \Delta_{\lambda} (kx)dk , \quad \lambda\in \mathbb C^n$$
and depend only on the eigenvalues of their argument. 
More precisely, for $ y \in \Omega$ and $s\in \mathbb C^n$ as above, 
\begin{equation}\label{spherical_Laplace} \int_\Omega e^{-(y\vert x)}\varphi_s(x) \Delta(x)^{-m/n} dx \,=\, \Gamma_\Omega(s) \varphi_s(y^{-1}).	
\end{equation}
For parameters $\lambda\in \Lambda_n^+$, the $\varphi_\lambda$ are just the spherical polynomials of $\Omega$ given by
 $$ \varphi_\lambda(x)= \frac{C_\lambda(\text{spec}(x), \tfrac{2}{d})}{C_\lambda(\underline 1,\tfrac{2}{d})}.$$
 Rewriting \eqref{spherical_Laplace} by means of identity \eqref{Bessel_Laplace_cone}, one gets formula \eqref{Mac_C} for the particular multiplicities $k=d/2.$
 It is well-known that the spherical functions of $\Omega$ can be expressed in terms of Heckman-Opdam hypergeometric functions of type $A_{n-1}$ and with multiplicity $k=d/2,$ see e.g. \cite{RKV13} for cones of positive definite matrices over $\mathbb R, \mathbb C, \mathbb H.$   
 
 In the present paper, we shall establish a generalization of formula \eqref{spherical_Laplace} to the Dunkl setting of type $A_{n-1}$ with arbitrary multiplicity $k\geq 0$. Namely, we obtain in Corollary \ref{Cherednik_reform}
 the following Laplace transform identity for Heckman-Opdam hypergeometric functions of type $A_{n-1}$ and, more generally, for the associated Opdam-Cherednik kernel $\mathcal G_k$: For 
$ \lambda \in \mathbb C^n$ with $\text{Re}\,\lambda\geq \mu_0\cdot\underline 1$ and $z\in \mathbb C^n$ with $\text{Re}\, z >0$,
\begin{equation}\label{Laplace_G}
\int_{\R_+^n} E_k(-z,x)\,\mathcal{G}_k(\lambda,x)\,\Delta(x)^{-\mu_0-1}\omega_k(x)dx \,= \,
\Gamma_n(\lambda+\rho) \, \mathcal{G}_k(\lambda,\tfrac{1}{z}).
\end{equation}
  The first step towards the proof of \eqref{Laplace_G} will be a  rigorous proof of Macdonald's Conjecture (C). 
  More generally, we shall prove Dunkl-Laplace transform identities for the non-symmetric Jack polynomials in the sense
 of \cite{Opd95, KS97}, from which \eqref{Mac_C} then follows by symmetrization. 
These non-symmetric identities were already stated in \cite{BF98},  but the proof given there in terms of Laguerre expansions is involved and not fully carried out.  The proof we are presenting here is completely different and very natural; it is based on a reformulation via  Dunkl operators and is carried out by induction, using the raising operator of Knop and Sahi \cite{KS97} for the non-symmetric Jack polynomials.  The statement for the Cherednik kernel is then obtained via analytic continuation with respect to the spectral variable, and for the hypergeometric function it follows by symmetrization.  

\smallskip
Based on the Laplace transform identities for Jack polynomials, we then study hypergeometric series in terms of Jack polynomials of the form
$${}_pF_q(\mu;\nu;z,w) \,:= \sum\limits_{\lambda \in \Lambda_n^+} \frac{[\mu_1]_{\lambda}\cdots [\mu_p]_{\lambda}}{[\nu_1]_{\lambda}\cdots [\nu_q]_{\lambda}} \,\frac{C_\lambda(z)C_\lambda(w)}{\abs{\lambda}!\,C_\lambda(\underline{1})}   \quad (\mu\in \mathbb C^p, \nu\in \mathbb C^q)$$
as well as their non-symmetric analogues, 
and we establish Laplace transform identities between them. This generalizes known results on symmetric cones
and settles several conjectural Laplace transform formulas in \cite{Mac13}. 
As a further application, we finally prove a 
Post-Widder inversion theorem for the Dunkl-Laplace transform, which 
 complements a result by Faraut and Gindikin in
\cite{FG90} for the Laplace transform on symmetric cones.

\smallskip
The organization of this paper is as follows: Section 2 provides the necessary  background on  the type $A$ Dunkl setting, both in the rational and trigonometric case. In Section 3, we collect results on the symmetric and non-symmetric Jack polynomials which will be relevant in the sequel, and we prove the Dunkl-Laplace transform identities for Jack polynomials. Section 4 contains a digression on some useful properties of the Opdam-Cherednik kernel for arbitrary root systems. These will be employed (for type $A_{n-1}$)  in Section 5, where the Laplace transform identities for Jack polynomials are extended to the Opdam-Cherednik kernel and to the hypergeometric function. Section 6 is devoted to the study of 
Jack-hypergeometric series, and Section 7 contains the Post-Widder inversion formula in the Dunkl setting.

\smallskip
To avoid notational overload, we shall always suppress in our notations the dependence on the fixed multiplicity parameter $k\geq 0 $ on the root system $A_{n-1}$.

\section{The type A Dunkl setting}\label{Dunkl}

For a general background, the reader is referred to \cite{Dun89, dJ93, R03, DX14, Opd95, R20, HO21}.

We consider the root system $A_{n-1}= \{\pm(e_i-e_j): 1 \leq i<j\leq n\}$ in the Euclidean space $\mathbb R^n$ with inner product $\langle x,y\rangle = \sum_{i=1}^nx_iy_i$  and norm $|x|= \sqrt{\langle x,x \rangle}$, where the $e_i$ denote the standard basis vectors. The inner product $\langle\m,\m\rangle$ is extended to $\mathbb C^n$ in a bilinear way. 
The reflection group generated by $A_{n-1}$ is the symmetric group $\mathcal{S}_n$ on $n$ elements. It acts on functions $f:\mathbb R^n \to \mathbb C$ by $\sigma f:= f(\sigma^{-1}\cdot).$ 
The (rational) Dunkl operators  associated with $A_{n-1}$ and a multiplicity parameter $k\in \mathbb C$ are given by
$$T_\xi = T_\xi(k) =\partial_\xi + \frac{k}{2}\cdot\!\sum_{\alpha \in  A_{n-1}} \langle \alpha, \xi \rangle \frac{1-s_\alpha}{\langle \alpha, x\rangle}  \quad (\xi \in \mathbb R^n)$$
where $s_\alpha$ denotes the orthogonal reflection in the root $\alpha$. 
For fixed $k$, the operators $T_\xi(k), \,  \xi \in \mathbb R^n$ commute. Let $\mathcal P:=\mathbb C[\mathbb R^n]$ denote the space of polynomial functions on $\mathbb R^n$. The assignment $\,\langle\m,\xi\rangle \mapsto T_\xi$ has a unique extension to a unital morphism of algebras $\phi: \mathcal P \to \text{End}(\mathcal P)$, and we shall write $\phi(p)=:p(T).$ 

The Dunkl operators satisfy a product rule: if $f, g\in C^1(\mathbb R^n)$ and one of them is symmetric, then $T_\xi(fg) = (T_\xi f)g + fT_\xi g.$ 

In this paper, we shall always assume that $k\geq 0.$
Let $E(\lambda, \m) := E_k(\lambda,\m):\C^n\to \C$, $\lambda \in \C^n$ be the associated Dunkl kernel  of type $A_{n-1},$ i.e. the unique holomorphic solution of the system
$$\begin{cases}
T_\xi E(\lambda,\m)=\braket{\lambda,\xi}E(\lambda,\m) , \quad \text{ for all } \xi \in \R^n \\
\; \, \; E(\lambda,0) =1
\end{cases}.$$
The associated Bessel function is defined by
$$J(\lambda,z) := J_k(\lambda,z) := \frac{1}{n!}\sum_{\sigma \in \mathcal{S}_n} E(\lambda,\sigma z).$$
The Dunkl kernel is positive on $\R^n\times \R^n$ and for $\lambda,z \in \C^n, s \in \C, \sigma \in \mathcal{S}_n$ one has
$$E(\lambda,z)=E(z,\lambda), \quad  E(s\lambda,z)=E(\lambda,sz), \quad E(\sigma\lambda,\sigma z)=E(\lambda,z).$$
Moreover, it satisfies
$$ E(\lambda, z+\underline s) = e^{\langle \lambda, \underline s\rangle} \cdot E(\lambda, z)$$
for  $s\in \mathbb C,$ where
\[ \underline s := s\cdot(1,\ldots , 1) \in \mathbb C^n.\]
Let 
$$ \omega(x) = \prod_{1\leq i<j\leq n} |x_i-x_j|^{2k}. $$
Dunkl analysis associated with the root system $A_{n-1}$ generalizes the radial analysis on symmetric cones, which just corresponds to the multiplicity values $k=d/2$, where $d$ is the Peirce constant of the cone. There is a well-behaved Laplace transform of functions $f\in L_{loc}^1(\mathbb R_+^n),$ for arbitrary $k\geq 0$,  which is given by
$$ \mathcal Lf(z) = \int_{\mathbb R_+^n} f(x) E(-z,x)\omega(x)dx,$$
and was first considered in \cite{BF98}. 
For $x\in \mathbb R_+^n$ and $z\in \mathbb C^n$ with $\,\text{Re}\, z \geq a$ for some $a\in \mathbb R^n$ (which is understood componentwise),  the type $A$ Dunkl kernel satisfies the exponential bound 
\begin{equation}\label{DunklKernelEstimate}
|E(-z,x)| \leq \, \exp\bigl(-\|x\|_1\cdot \min_{1\leq i\leq n}a_i\bigr),\,
\end{equation}
see \cite{R20}. Here $\nrm{x}_1= \sum_{i=1}^n |x_i|.\,$ 
This estimate, which seems to be exclusive in type $A$, guarantees good convergence properties of the Laplace integral. In particular, we recall the following

\begin{lemma}[\cite{R20}] \label{Laplaceconv} Suppose that $f: \mathbb R_+^n \to \mathbb C$ is measurable and exponentially bounded according to $\, |f(x)| \leq C e^{s\|x\|_1} $ with some constants $C>0$ and $s\in \mathbb R$. Then $\mathcal Lf(z)$ exists and is holomorphic on $\{z\in \mathbb C^n: \text{Re}\, z >\underline{s}\}.$ Moreover, for each polynomial $p \in \mathcal P,$ $\, p(-T)(\mathcal Lf) = \mathcal L(fp)$ on $\{\text{Re} \, z >\underline s\}.$
	
\end{lemma}

Let us turn to the trigonometric setting. Here root systems are required to be spanning, and therefore 
we consider $A_{n-1}$ as a subset of 
$$\mathbb R_0^n := \{x\in \mathbb R^n: x_1 + \ldots + x_n =0\}.$$ 
We fix the positive subsystem $A_{n-1}^+ = \{ e_j-e_i: i<j\}.$ 
The (trigonometric) Cherednik operators on $\mathbb R_0^n$ associated with $A_{n-1}^+$
and a multiplicity parameter $k\geq 0$ are 
given by
$$D_\xi= D_\xi(A_{n-1}^+,k):= \partial_\xi-\braket{\rho(R_+),\xi}+k\sum\limits_{\alpha \in A_{n-1}^+} \braket{\alpha,\xi} \frac{1-s_\alpha}{1-e^{-\braket{\,\m,\alpha}}}, \quad \xi \in \R^n,$$
with the Weyl vector
$$\rho(R_+)=\rho(R_+, k):= \,\frac{k}{2}\sum\limits_{\alpha \in A_{n-1}^+} \!\alpha = \,-\frac{k}{2}(n-1,n-3,\ldots, -n+1).$$
The operators $D_\xi\,,\, \xi \in \R_0^n$ commute. Let $\mathbb C_0^n = \mathbb R_0^n \oplus i\mathbb R_0^n.$ 
Due to \cite{Opd95}, there exist an $\mathcal{S}_n$-invariant tubular open neighborhood $U$ of $\mathbb R_0^n$ in $\mathbb C_0^n$  and a unique holomorphic function 
$G  = G_k $ on $\mathbb C_0^n\times (\mathbb R_0^n + iU)$  which satisfies the joint eigenvalue problem
$$
\begin{cases}
D_\xi G(\lambda,\m)=\braket{\lambda,\xi}G(\lambda,\m), \quad \text{ for all } \xi \in \R_0^n \\
\; \; \; \, G(\lambda,0)=1.
\end{cases}$$
The function $G$ is called the Opdam-Cherednik kernel associated with $A_{n-1}$ and multiplicity $k.$ By symmetrization, one obtains the associated Heckman-Opdam hypergeometric function 
$$F(\lambda,z) = F_k(\lambda,z) = \frac{1}{n!}\sum\limits_{\sigma \in \mathcal{S}_n} G_k(\lambda,\sigma z).$$
It is actually $\mathcal S_n$-invariant in both $\lambda$ and $z$.
According to \cite[Theorem 13.15]{KO08} (see also \cite[Cor.~8.6.2]{HO21}), $F$ extends to a holomorphic function on $ \mathbb C_0^n \times(\mathbb R_0^n + i\Omega)$ with
$$\Omega = \{x \in \mathbb R_0^n: |x_i-x_j| < \pi \quad \text{ for all } 1\leq i < j \leq n\},$$
and the proof of \cite[Theorem 3.15]{Opd95} shows that $G$ extends holomorphically to the same domain.
We mention that 
For $k=d/2$ with $d=1,2,4,$ the functions $t\mapsto F(\lambda,2t)$ on $\mathbb R_0^n$ are naturally identified with the spherical functions of $SL_n(\mathbb F)/SU_n(\mathbb F)$ with $\mathbb F= \mathbb R, \mathbb C, \mathbb H$; cf. \cite{Opd95, RKV13}. 

In this paper, we shall work with natural extensions of $G$ and $F.$ In order to define them, we consider the Cherednik operators $D_\xi$ as operators on $\mathbb R^n$, for arbitrary $\xi\in \mathbb R^n.$

\begin{lemma}\label{Productrule}
Consider the orthogonal projection $\pi:\mathbb R^n \to \mathbb R_0^n, \, x\mapsto x-\frac{1}{n}\langle x, \underline 1\rangle\cdot \underline 1. $ Then for all $\,\xi \in \mathbb R^n$ and $f, g\in C^1(\mathbb R^n),$ 
\begin{enumerate}\itemsep=+1pt
\item[\rm{(1)}] $D_{\xi}(f\circ \pi)=D_{\pi(\xi)}f \circ \pi$.
\item[\rm{(2)}] If $f$ or $g$ is $\mathcal{S}_n$-invariant, then $D_\xi (fg)=f(D_\xi g) + (D_\xi f)g+\braket{\rho(R_+),\xi}fg$.
\end{enumerate}
\end{lemma}

\begin{proof} Part (1) is obtained by a short calculation, using that $\partial_{\xi}(f\circ \pi)=\partial_{\pi(\xi)}f \circ \pi$ and that $\rho$ and all roots are contained in $\mathbb R_0^n.$ 
Part (2) is straightforward. 
\end{proof}

In order to extend the Opdam-Cherednik kernel, put $V:= \Omega + \mathbb R\underline 1 \subset \mathbb R^n,$ which is $ \mathcal S_n$-invariant with $\pi(V)=\Omega.$ 
Keeping the notation, we define
\begin{equation}\label{G_extended} G: \mathbb C^n\times (\mathbb R^n+iV)\to \mathbb C, \,\, G(\lambda,z) := e^{\frac{1}{n}\langle z,\underline 1\rangle\langle\lambda,\underline 1\rangle}\, G(\pi(\lambda),\pi(z)),\end{equation}
where  $\pi$ is canonically extended to a linear mapping $\pi: \mathbb C^n\to \mathbb C_0^n.$ 
The extended hypergeometric function $F$ is obtained in the same way. 
\begin{proposition}\label{extendedchar} The extended Opdam-Cherednik kernel 
$G(\lambda,\m), \lambda \in \mathbb C^n$ is the unique holomorphic solution on $\mathbb R^n +iV$ to the system
\begin{equation}\label{CherednikSystem}
\begin{cases}
D_\xi G(\lambda,\m) = \langle\lambda,\xi\rangle G(\lambda,\m)  & \text{ for all }\, \xi \in \mathbb R^n \\
\quad G(\lambda,0)=1.

\end{cases}
\end{equation}
\end{proposition}

\begin{proof} Since $z \mapsto e^{\langle z,\underline{1}\rangle \langle\lambda,\underline{1}\rangle/n}\,$ is $\mathcal S_n$-invariant, it is immediate from Lemma \ref{Productrule}(2) that $G(\lambda,\m)$ solves  \eqref{CherednikSystem}.
Assume that $f$ is a further solution in some neighborhood $U^\prime$ of $0 \in \mathbb C^n.$ Then for $\xi \in \mathbb R_0^n$, Lemma \ref{Productrule} (1) gives 
$$D_\xi(f\circ \pi) =D_\xi f \circ \pi = \braket{\lambda,\xi}(f\circ \pi)=\braket{\pi(\lambda),\xi}(f\circ \pi).$$
The uniqueness of the Opdam-Cherednik kernel on $\mathbb R_0^n$ thus implies that $f$ coincides with $ G(\pi(\lambda),\m)$ on $\pi(U')$. Finally, for $\xi=\underline{1}$ we have $D_{\underline{1}}=\partial_{\underline{1}}$, and the eigenvalue equation $D_{\underline{1}}f = \langle\lambda,\underline{1}\rangle f\,$ leads to
$$\partial_{\underline{1}}\bigl(e^{-\braket{\m,\underline{1}}\braket{\lambda,\underline{1}}/n} \m f\bigr) \equiv 0 \quad \text{on } \mathbb R^n. $$
Therefore $\,x \mapsto e^{-\tfrac{\braket{x,\underline{1}}\braket{\lambda,\underline{1}}}{n}}f(x)$ is in each point constant in direction $\underline{1}$, i.e.
$$e^{-\tfrac{\braket{x,\underline{1}}\braket{\lambda,\underline{1}}}{n}}f(x) = e^{-\tfrac{\braket{\pi(x),\underline{1}}\braket{\lambda,\underline{1}}}{n}}f(\pi(x))=f(\pi(x))=G(\pi(\lambda),\pi(x)).$$
The result follows by analytic extension.
\end{proof}

\section{Jack polynomials and Macdonald's conjecture}\label{Macdonalds_Conjecture}
	
	We first recall some well-known facts about Jack polynomials from \cite{KS97, For10, Sta89}. 
	 Let $\Lambda_n^+=\{\lambda\in \mathbb N_0^n: \lambda_1 \geq \ldots \geq \lambda_n\}$  denote the set of partitions of length at most $n$. The dominance order on $\Lambda_n^+$ is given by 
	 $$\mu \leq_D \lambda \,\, \text{ iff } \,\,|\lambda|=|\mu| \,\text{ and  } \sum_{j=1}^r \mu_j \leq \sum_{j=1}^r \lambda_j \text{ for all } r=1,\ldots,n\, ,$$
where $|\lambda|=\lambda_1+\ldots+\lambda_n$. The dominance order is extended from $\Lambda_n^+$ to  $\mathbb N_0^n$ as follows: 
For each composition $\eta \in \mathbb N_0^n$ denote by $\eta_+ \in \Lambda_n^+$ the unique element in the $\mathcal S_n$-orbit of $\eta$. Then the dominance order on $\mathbb N_0^n$ is defined by
$$\kappa \preceq \eta \,\,\text{ iff }\,\,  \begin{cases}
\kappa_+\le_D \eta_+ \,,& \kappa_+\neq \eta_+ \\
w_\eta \le  w_\kappa \,,& \kappa_+=\eta_+
\end{cases},$$
where $w_\eta \in \mathcal S_n$ is the shortest element with $w_\eta \eta_+=\eta,$ and $\le$ refers to the Bruhat order on $\mathcal S_n$.	 
Consider the rational Cherednik operators
$$\mathcal D_j = \mathcal D_j(k):= x_jT_j +k(1-n) + k\sum_{i>j} s_{ij}, \quad 1=1, \ldots , n$$
where the $T_j := T_{e_j}(k)$ are the type $A$ Dunkl operators with multiplicity $k$ and $s_{ij}$ denotes  the reflection in the root $e_i-e_j$, which acts by interchanging $x_i$ and $x_j$. We remark that our notion differs by a factor $k$ from that in \cite{For10}. This facilitates the handling of the case $k=0.$  The operators $\mathcal D_j$ are closely related to the usual Cherednik operators $D_j := D_{e_j}(k)$. Indeed, consider 
 $f\in C^1(U)$ for some open $U \subseteq \mathbb R^n$ and define $g: \exp^{-1}(U)\subseteq \mathbb R^n\to \mathbb C$ by $g(x) := f(e^x),$ where $e^x$ is understood componentwise. Then a short calculation gives
\begin{equation}\label{Cherednik_connection} \bigl(D_j - \frac{k}{2}(n-1)\bigr) g(x) = (\mathcal D_jf)(e^x).\end{equation}

The operators $\mathcal D_j$ are upper triangular with respect to $\preceq$  on $\mathcal P=\mathbb C[\mathbb R^n].$ More precisely, 
$$ \mathcal D_j x^\eta\,=\, \overline \eta_j x^\eta + \sum_{\kappa \prec \eta} d_{\kappa \eta} x^\kappa$$
with some $d_{\kappa\eta}\in \mathbb R$ and 
$$\overline{\eta}_j=\eta_j-k\#\set{i<j \mid \eta_i \ge \eta_j}-k\#\set{i>j \mid \eta_i>\eta_j}.$$ 
	 The non-symmetric Jack polynomials of index $\alpha = 1/k$ with $k \in [0, \infty)$ can be characterized as the unique basis $\bigl(E_\eta = E_\eta(\,.\,;\alpha\bigr))_{\eta\in \mathbb N_0^n}$ of $\mathcal P$ satisfying
	 \begin{enumerate}\itemsep=+1pt
\item[\rm{(1)}] $E_\eta(x)=x^\eta + \sum_{\kappa \prec \eta} c_{\eta\kappa }x^\kappa\,$ with $\,c_{\kappa \eta} \in \mathbb C$,
\item[\rm{(2)}] $\mathcal D_jE_{\eta}=\overline{\eta}_jE_{\eta}\,$ for all $j=1,\ldots,n$.
\end{enumerate} 
By definition, $E_\eta$ is homogeneous of degree $|\eta|=\eta_1 + \ldots +\eta_n$, 
and for $k=0$ we have $E_\eta(x;\infty)=x^\eta$.	
Property (2) together with Proposition \ref{extendedchar} and identity \eqref{Cherednik_connection} show 
that the polynomials $E_\eta$ are related to the (extended) Opdam-Cherednik kernel via  
\begin{equation}\label{connection_kernelpolys} \frac{E_\eta\bigl(e^x\bigr)}{E_\eta(\underline 1)} \, = \, G\bigl(\overline \eta + \frac{k}{2}(n-1)\underline 1\,, x\bigr), \quad \overline \eta = (\overline \eta_1, \ldots, \overline \eta_n).
\end{equation}
Following \cite{For10}, we denote by $P_\lambda(x) = P_\lambda(x;\alpha), \,\,\lambda\in \Lambda_n^+$  the symmetric Jack polynomials in $n$ variables of index $\alpha = \frac{1}{k}$ in monomial normalization. In the limiting case $k=0$, they coincide with the monomial symmetric functions $$m_\lambda(x) = \sum_{\eta \in \mathcal S_n \lambda} x^{\eta}.$$ The non-symmetric and symmetric Jack polynomials of the same index are related via symmetrization: for $\lambda\in \Lambda_n^+ $ and $\eta \in \mathbb N_0^n$ with $\eta_+=\lambda,$ 
\begin{equation}\label{Symm_1} \frac{P_\lambda(x)}{P_\lambda(\underline 1)} = \frac{1}{n!} \sum_{\sigma\in \mathcal S_{n}} \frac{E_{\eta}(\sigma x)}{E_{\eta}( \underline 1)}.\end{equation}
The Jack polynomials $P_\lambda$ satisfy a binomial formula:
\begin{equation}\label{binom} \frac{P_\lambda(\underline{1}+x)}{P_\lambda(\underline{1})} = \sum\limits_{\mu \tm \lambda} \binom{\lambda}{\mu} \frac{P_\mu(x)}{P_\mu(\underline{1})}\,,\end{equation}
where $\mu \tm \lambda$ for $\lambda, \mu \in \Lambda_n^+$ means $\mu_i \le \eta_i$ for all $i,$ and $\binom{\lambda}{\mu}=\binom{\lambda}{\mu}_{\!k} \geq 0 $ is a generalized binomial coefficient. 
Symmetrization in \eqref{connection_kernelpolys}   yields a relation between the (extended) hypergeometric function $F=F_k$ and the symmetric Jack polynomials:
If  $\lambda\in \Lambda_n^+$, then 
\begin{equation}\label{lambda-rho-id} \overline \lambda + \frac{k}{2}(n-1)\cdot \underline 1 \,=\, \lambda - \rho\end{equation}
and therefore 
\begin{equation}\label{connection_hypergeopolys} \frac{P_\lambda\bigl(e^x\bigr)}{P_\lambda(\underline 1)}  \,=\, F(\lambda -\rho, x).\end{equation}

In the following lemma,  we collect some further useful properties of the non-symmetric Jack polynomials $E_\eta = E_\eta(\,.\,;\tfrac{1}{k})$, which can be found in \cite{For10, KS97, S98}  for $k>0$ and are obvious for $k=0.$ 
 Here we consider the Jack polynomials as functions on $\mathbb C^n$. 

\begin{lemma}\label{NonSymmetricJackProperties}
\begin{enumerate}\itemsep=-1pt
\item[\rm{(1)}] For all $p \in \mathbb N_0$, 
$$\Delta(z)^p E_\eta(z)=E_{\eta+\underline{p}}(z).$$

By this property, the non-symmetric Jack polynomials uniquely extend to indices $\eta \in \mathbb Z^n$.
\item[\rm{(2)}] Let $z \in \mathbb C^n$ with $z_i \neq 0$ for all $\,i=1,\ldots,n.$ Then
$$E_\eta\left(\tfrac{1}{z}\right) = E_{-\eta^R}(z^R),$$
where $\eta^R = (\eta_n, \ldots, \eta_1), \, z^R = (z_n, \ldots, z_1).$
\item[\rm{(3)}] Let $\Phi$ be the so-called raising operator, which acts on functions $f: \mathbb C^n\to \mathbb C$ by 
$$\Phi f(z)=z_nf(z_n,z_1,\ldots,z_{n-1})$$
and on $\mathbb N_0^n$ by 
$$\Phi \eta=(\eta_2,\ldots,\eta_n,\eta_1+1).$$ Then the non-symmetric Jack polynomials satisfy
$$\Phi E_\eta= E_{\Phi \eta}  \quad (\eta\in \mathbb N_0^n). $$
According to part (1) this identity extends to all $\eta \in \mathbb Z^n$, because $\Phi(\Delta E_\eta)=\Delta \m \Phi E_\eta$ and $\Phi(\eta+\underline{p})=\Phi(\eta)+\underline{p}\,$ for all $p \in \mathbb N.$
\item[\rm{(4)}] The coefficients $c_{\eta\kappa}$ in the monomial expansion of $E_\eta,\eta \in \mathbb N_0^n,$ are non-negative.
\end{enumerate}
\end{lemma}

\begin{lemma}\label{Polys_eins}
There exists a polynomial $Q\in \mathcal P$ such that for alle $\eta \in \mathbb N_0^n$ and $\lambda\in \Lambda_n^+,$ 
$$ 0 \leq E_\eta(\underline 1) \leq Q(\eta), \quad 0 \leq P_\lambda (\underline 1) \leq Q(\lambda). $$ 
\end{lemma}

 \begin{proof} 
 By \cite[Prop. 12.3.2]{For10}, 
$$ E_\eta(\underline 1) = \prod_{(i,j) \in \eta} \frac{j + kn - k\ell^\prime(\eta, i,j)}{\eta_i - j+1 + k\ell(\eta,i,j)+k}$$
with the leg length and coleg length $\ell(\eta,i,j), \ell^\prime(\eta, i, j) \in \{0, \ldots, n\}$. Therefore
$$ E_\eta(\underline 1) \leq  \, \prod_{i=1}^n \prod_{j=1}^{\eta_i}  
\frac{j+kn}{\eta_i -j + 1} = \prod_{i=1}^n \frac{(\eta_i +kn)!}{(kn)!\,\eta_i!}\,,$$ which is polynomially bounded in $\eta$ by Stirling's formula. Similarly (c.f. \cite[Prop. 12.6.2]{For10}),
 $$ P_\lambda(\underline 1)  = \prod_{(i,j) \in \lambda} \frac{
j-1+ kn - kl^\prime(\lambda, i,j)}{\lambda_i - j + kl(\eta,i,j)+k}\, \leq \prod_{i=1}^n \prod_{j=1}^{\lambda_i} \frac{j-1+kn}{\lambda_i - j+k}\,,$$
 which is also polynomially bounded in $\lambda$. 
\end{proof}

\noindent
To formulate the main results of this section, we 
introduce the gamma function 
\begin{equation}\label{Gamma_n} \Gamma_n(\lambda)= \Gamma_n(k;\lambda) := d_n(k) \cdot \prod_{j=1}^n \Gamma(\lambda_j -k(j-1)) \quad (\lambda \in \mathbb C^n), \end{equation}
 with $$ d_n(k) := \prod_{j=1}^n \frac{\Gamma(1+jk)}{\Gamma(1+k)},$$
 as well as the generalized Pochhammer symbol
 $$ [\mu]_\eta := \prod_{j=1}^n (\mu-k(j-1))_{\eta_j}\, =\, 
 \frac{\Gamma_n(\underline\mu + \eta)}{\Gamma_n(\underline\mu)} \,\quad (\mu\in \mathbb C, \eta\in \mathbb N_0^n).$$
 For abbreviation, we also write
 $$ \Gamma_n(\mu):= \Gamma_n(\underline \mu) \quad \text{ for } \mu \in \mathbb C.$$
 Note that $\Gamma_n$ differs by the factor $d_n(k)$ from the notion in \cite{R20, Mac13}, but is in accordance with the notion for the gamma function on symmetric cones. 
We shall obtain the master theorem  as a consequence of the following result, which  involves the type $A$ Dunkl operators $T=T(k)$ with multiplicity $k$. 
 
 \begin{theorem}\label{main_1} Fix $k\geq 0,$ and consider the non-symmetic Jack polynomials  $(E_\eta)_ {\eta\in \mathbb N_0^n} $ and the symmetric Jack polynomials $(P_\lambda)_{\lambda\in \Lambda_n^+}\,$ of index $1/k\,.$ 
 Then for all $\mu\in \mathbb C$ and all $x\in \mathbb R^n$ with $x_i\not = 0$ for all $i=1, \ldots, n$, \parskip=1pt
 \begin{enumerate}\itemsep=+1pt
 \item[\rm{(1)}] $\displaystyle E_\eta(T)\Delta^{-\mu}(x) = (-1)^{|\eta|}\,[\mu]_{\eta_+}E_\eta\left(\tfrac{1}{x}\right)\Delta(x)^{-\mu}$;
\item[\rm{(2)}] $\displaystyle P_\lambda(T)\Delta^{-\mu}(x) = (-1)^{|\lambda|}\,[\mu]_{\lambda}\, P_\lambda\left(\tfrac{1}{x}\right)\Delta(x)^{-\mu}$.	\end{enumerate}	
 \end{theorem}
 
 \noindent
 For the proof, we first note 
 
 \begin{lemma} The set $\mathbb N_0^n$ can be recursively constructed from $0 \in \mathbb N_0^n$ by a chain of the following operations:
 \begin{enumerate}\itemsep=+1pt
\item[\rm{(i)}] apply the raising operator $\Phi$ to  $\eta \in \mathbb N_0^n$,
\item[\rm{(ii)}] apply  a simple permutation $s_i=(i,i+1)$ to  $\eta \in \mathbb N_0^n$ with $\eta_i<\eta_{i+1}$.
\end{enumerate} 	
 \end{lemma}
\begin{proof}  This is easily verified by induction on the weight $|\eta|.$
Indeed, assume that all elements of weight at most $r$ are already constructed and take $\eta \in \mathbb N_0^n$ with $|\eta|=r+1$. Consider the maximal index $j=1,\ldots,n$ with $\eta_j \neq 0$ and $\eta_k=0$ for $j<k\le n$. Then
$$\eta=(\eta_1,\ldots,\eta_j,0,\ldots,0) = s_j\cdots s_{n-1}(\eta_1,\ldots,\eta_{j-1},0,\ldots,0,\eta_j)=s_j\cdots s_{n-1}\Phi \widehat{\eta}$$
with
$\,\widehat{\eta}=(\eta_j-1,\eta_1,\ldots,\eta_{j-1},0,\ldots,0),$ which is already constructed by induction hypothesis.
\end{proof}

 \begin{proof}[Proof of Theorem \ref{main_1}] 
 Part (2) is obtained from (1) by symmetrization. Part (1) is clear for $\eta = 0,$ since $E_0=1.$ 
 In view of the above observation,  it therefore suffices to consider the following two cases:
 
 \smallskip\noindent
\emph{Case 1.}  Assume formula (1) is correct for some $\eta \in \mathbb N_0^n$ with $\eta_i<\eta_{i+1},$ and consider $E_{s_i\eta}.$ According to \cite[Proposition 12.2.1]{For10} there exists a constant $d_i^\eta \in \R$ such that
$$E_{s_i\eta}= d_i^\eta E_\eta + s_iE_\eta.$$
The Dunkl operators are $\mathcal{S}_n$-equivariant, i.e. $\sigma T_\xi\sigma^{-1}=T_{\sigma \xi}$, $\sigma \in \mathcal{S}_n$. Hence the symmetry of $\Delta(x)$ leads to
\begin{align*}
(s_iE_\eta)(T)\Delta(x)^{-\mu} &= (s_iE_\eta(T) (s_i\Delta)^{-\mu})(x) = E_\eta(T)\Delta^{-\mu}(s_ix) \\
&= (-1)^{|\eta|} \,[\mu]_{\eta_+} E_\eta\bigl(\tfrac{1}{s_ix}\bigr)\Delta(s_ix)^{-\mu} \\
&= (-1)^{|\eta|}\, [\mu]_{\eta_+} (s_iE_\eta)\left(\tfrac{1}{x}\right)\Delta(x)^{-\mu} 
\end{align*}
As $|s_i\eta|=|\eta|$ and $(s_i\eta)_+=\eta_+,$ the formula follows for $s_i\eta$ by linear combination.

\smallskip\noindent\emph{Case 2.}  Assume that  formula (1) is correct for some $\eta \in \mathbb N_0^n$, and consider $\Phi\eta.$  Using the identity $\Phi E_\eta=E_{\Phi\eta}$ from Lemma \ref{NonSymmetricJackProperties} and  the product rule for the Dunkl operators, we calculate 
\begin{align}\label{MasterformelCalculation1}
E_{\Phi\eta}(T)\Delta(x)^{-\mu} &= T_nE_\eta(T_n,T_1,\ldots,T_{n-1})\Delta(x)^{-\mu} \notag\\
&= T_n\Big( (-1)^{|\eta|}\,[\mu]_{\eta_+} E_\eta(\tfrac{1}{x_n},\tfrac{1}{x_1},\ldots,\tfrac{1}{x_{n-1}})\Delta(x)^{-\mu}\Big) \notag\\
&= (-1)^{|\eta|}[\mu]_{\eta_+} \Big((T_n\,\Delta(x)^{-\mu})E_\eta(\tfrac{1}{x_n},\tfrac{1}{x_1},\ldots,\tfrac{1}{x_{n-1}}) \notag\\
&\quad + \Delta(x)^{-\mu}(T_n E_\eta(\tfrac{1}{x_n},\tfrac{1}{x_1},\ldots,\tfrac{1}{x_{n-1}}))\Big). 
\end{align}
As $T_n$ acts on symmetric functions as the partial derivative $\tfrac{\partial}{\partial x_n}$, we have $$T_n\Delta(x)^{-\mu}=-\mu\, x_n^{-1}\Delta(x)^{-\mu}\,.$$ Parts (1) and (2) of Proposition \ref{NonSymmetricJackProperties} show that
$$E_\eta(\tfrac{1}{x_n},\tfrac{1}{x_1},\ldots,\tfrac{1}{x_{n-1}})=\Delta^{-p}(x)E_{\eta^*}(x_{n-1},\ldots,x_1,x_n) $$
with $\,\eta^*=-\eta^R+\underline{p},$ where $p \in \mathbb N$ is so large that $-\eta^R+\underline{p} \in \mathbb N_0^n$. Note further that $\tfrac{1}{x_n}E_{\eta}(\tfrac{1}{x_n},\tfrac{1}{x_1},\ldots,\tfrac{1}{x_{n-1}})=E_{\Phi\eta}(\tfrac{1}{x}).$
Thus formula \eqref{MasterformelCalculation1} reduces  to
\begin{align}\label{MasterformelCalculation2}
& E_{\Phi\eta}(T)\Delta(x)^{-\mu} \notag \\
&=(-1)^{|\eta|}[\mu]_{\eta_+}\Delta(x)^{-\mu}\big( -\mu E_{\Phi\eta}(\tfrac{1}{x}) + T_n(\Delta^{-p}(x)E_{\eta^*}(x_{n-1},\ldots,x_1,x_n))\big) .
\end{align}
Again by the product rule for $T_n$ and that fact that $T_n$ commutes with $s_1,\ldots,s_{n-2},$ we further obtain
\begin{align*}
&T_n(\Delta^{-p}(x)E_{\eta^*}(x_{n-1},\ldots,x_1,x_n)) \\
&= -p\Delta(x)^{-p}E_{\eta^*}(x_{n-1},\ldots,x_1,x_n) + \Delta(x)^{-p}(T_nE_{\eta^*}(x_{n-1},\ldots,x_1,x_n)) \\
&=-px_n^{-1}\Delta(x)^{-p}E_{\eta^*}(x_{n-1},\ldots,x_1,x_n) + \Delta(x)^{-p}(T_nE_{\eta^*})(x_{n-1},\ldots,x_1,x_n) \\
&= x_n^{-1}\Delta(x)^{-p}\Big(-pE_{\eta^*}(x_{n-1},\ldots,x_1,x_n) + (x_nT_nE_{\eta^*})(x_{n-1},\ldots,x_1,x_n)\Big).
\end{align*}
As $\, x_nT_n = \mathcal D_n +k(n-1),$ we have 
$$x_nT_nE_{\eta^*}(x_{n-1},\ldots,x_1,x_n)=(\overline{\eta^*}_n + n-1)E_{\eta^*}(x_{n-1},\ldots,x_1,x_n)$$
with $\,\overline{\eta^*}_n=\eta^*_n -k\#\set{\ell<n \mid \eta^*_\ell\ge \eta_n^*}$, so that
\begin{align*}
& T_n(\Delta^{-p}(x)E_{\eta^*}(x_{n-1},\ldots,x_1,x_n)) \\
&= \bigl(-p+(\overline{\eta^*}_n+k(n-1))\bigr)\Delta(x)^{-p}\tfrac{1}{x_n}E_{\eta^*}(x_{n-1},\ldots,x_1,x_n) \\
&=\bigl(-p+(\overline{\eta^*}_n+k(n-1))\bigr)\tfrac{1}{x_n}E_{\eta}(\tfrac{1}{x_n},\tfrac{1}{x_1},\ldots,\tfrac{1}{x_{n-1}}) \\
&=\bigl(-p+(\overline{\eta^*}_n+k(n-1))\bigr)E_{\Phi\eta}(\tfrac{1}{x}).
\end{align*}
Thus  equation \eqref{MasterformelCalculation2} reduces to
\begin{equation}\label{MasterformelCalculation3}
E_{\Phi\eta}(T)\Delta(x)^{-\mu} = (-1)^{|\eta|}[\mu]_{\eta_+} (-\mu -p+\overline{\eta^*}_n + k(n-1))E_{\Phi\eta}(\tfrac{1}{x})\Delta(x)^{-\mu}.
\end{equation}
Let $1\le j \le n$ be minimal such that the entry $j$ in $\eta_+$ is equal to $\eta_1$, i.e.
\begin{equation}\label{CountinLambdaPos}
j-1=\#\set{\ell>1 \mid \eta_1<\eta_\ell}.
\end{equation}
Now at position $j$ in $(\Phi\eta)_+$ there is $\eta_1+1$.  Thus, by definition of $j$ we have
\begin{align*}
\overline{\eta^*}_n &= (\eta^R+p)_n - k\#\set{\ell<n \mid -\eta_\ell^R+p \ge -\eta_n^R+p} \\
&=p+\eta_1-k\#\set{\ell<n \mid \eta_{n-\ell+1}\le \eta_1} \\
&=p+\eta_1-k\#\set{\ell>1 \mid \eta_{\ell} \le \eta_1} \\
&=p+\eta_1-k(n-j).
\end{align*}
So finally, since $j$ is the position of $\eta_1+1=(\Phi\eta)_n$ in $(\Phi\eta)_+$ we have that $(\Phi\eta)_+$ is exactly $\eta_+$ plus an $1$ at position $j$. Therefore
\begin{align*}
&(-1)^{|\eta|}[\mu]_{\eta_+} (-\mu -p+\overline{\eta^*}_n + k(n-1)) \\
&= (-1)^{|\eta|}[\mu]_{\eta_+} (-\mu + \eta_1-k(n-j) + k(n-1)) \\
&= (-1)^{|\eta|+1}[\mu]_{\eta_+} (\mu -k(j-1) +(\eta_1+1)-1) \\
&= (-1)^{|\Phi\eta|}[\mu]_{(\Phi\eta)_+}.
\end{align*}
Plugging this into \eqref{MasterformelCalculation3} we obtain the assertion.	
 \end{proof}

 \begin{theorem}[Laplace transform identities for Jack polynomials] \label{master}
 Let  $(E_\eta)_ {\eta\in \mathbb N_0^n} $ and $(P_\lambda)_{\lambda\in \Lambda_n^+}\,$  be the non-symmetric and symmetric Jack polynomials of index $1/k\,, k \geq 0.$ 
 Then for all $\mu\in \mathbb C$ with $\text{Re}\, \mu > \mu_0$ and $z \in \mathbb C^n$ with $\text{Re } z >0,$
 
 \begin{enumerate}\itemsep=-1pt
 \item[\rm{(1)}]	
  $\displaystyle \int_{\mathbb R_+^n} E(-z,x) E_\eta(x) \Delta(x)^{\mu-\mu_0-1} \omega(x)dx \,
     =\, \Gamma_n(\eta_+ + \underline\mu\,)E_\eta\bigl(\tfrac{1}{z}\bigr)\Delta(z)^{-\mu};$
 \item[\rm{(2)}]	  $\displaystyle \int_{\mathbb R_+^n} J(-z,x) P_\lambda(x) \Delta(x)^{\mu-\mu_0-1} \omega(x)dx \, =\, \Gamma_n(\lambda + \underline\mu\,)P_\lambda\bigl(\tfrac{1}{z}\bigr)\Delta(z)^{-\mu}.$
\end{enumerate}
 \end{theorem}
 
 Part (2) is just Macdonald's \cite{Mac13} Conjecture (C), and part (1) 
 corresponds to formula (4.38) in \cite{BF98} (there is a misprint: the Laguerre polynomial $E_\eta^{(L)}$ has to be replaced by $E_\eta\,).$ 
 
\begin{proof} The integrals converge by Lemma \ref{Laplaceconv}. 
	According to  \cite{R20}, 
$$ \Delta(z)^{-\mu} = \, \frac{1}{\Gamma_n(\mu)}\mathcal L\bigl(\Delta^{\mu-\mu_0-1}\bigr)(z),$$
and for each polynomial $p \in \mathcal P$, 
$$ p(-T) \Delta^{-\mu} (z) = \, \frac{1}{\Gamma_n(\mu)}\mathcal L\bigl(p\,\Delta^{\mu-\mu_0-1}\bigr)(z).$$
Now part (1) is immediate from Theorem \ref{main_1}\,(1) and part (2) follows by symmetrization. 
\end{proof}

  \section{Some general properties of the Opdam- Cherednik kernel}
  
  We want to extend the Laplace transform identities for Jack polynomials  to the Opdam-Cherednik kernel of type $A$. For this, we shall 
  need some relations for this  kernel which are of a general nature and seem not to be stated in the literature. In this section, we therefore consider an arbitrary crystallographic root system $R$  with Weyl group $W$  in a Euclidean space $\mathfrak a$. In particular, it is required that $R$ spans $\mathfrak a.$ The inner product in $\mathfrak a$ is extended to $\mathfrak a_\mathbb C$ in a bilinear way. A $W$-invariant function $k: R \to \mathbb C, \alpha \mapsto k_\alpha$ is called a multiplicity function. The Cherednik operators associated with multiplicity function $k$ and some  positive subsystem $R_+$ of $R$ are defined by
  $$ D_\xi(R_+) = D_\xi(R_+,k) = \partial_\xi - \langle \rho(R_+), \xi\rangle + \sum_{\alpha\in R_+} k_\alpha \langle \alpha, \xi\rangle \frac{1-s_\alpha}{1-e^{-\alpha}}.$$
  where $\, \rho(R_+) = \rho(R_+,k) = \frac{1}{2}\sum_{\alpha \in R_+} k_\alpha \alpha \,$ is the generalized Weyl vector. The $D_\xi(R_+),$ $ \xi \in \mathfrak a$ commute. 
According to \cite{Opd95}, there exist a $W$-invariant tubular neighborhood $U$ of $\mathfrak a $ in $\mathfrak a_{\mathbb C}$ such that for fixed multiplicity $k$ with $\text{Re}\,k\geq 0\,$ and each $\lambda \in \mathfrak a_\mathbb C$ there exists a unique holomorphic function $\,f=G(\lambda, \m) =G_k(R_+, \lambda, \m), $ called the Opdam-Cherednik kernel associated with $R_+$ and $k$, which satisfies the eigenvalue system
$$\begin{cases} D_\xi(R_+) f = \langle\lambda,\xi\,\rangle f \quad \text{for all } x\in \mathfrak a;\\
  f(0)=1.	
\end{cases}$$

\begin{proposition}\label{CherednikMinus}
The Cherednik operators  and the Opdam-Cherednik kernel associated with the positive system $R_+$ of $R$ have the following properties. 
\begin{enumerate}\itemsep=+1pt
\item[\rm{(1)}] $wD_\xi(R_+)w^{-1}=D_{w\xi}(wR_+)$ for all $w \in W$.
\item[\rm{(2)}] $D_\xi(R_+)f^- = -(D_\xi(R_-)f)^-,$ where $f^-(x)=f(-x)$ and $\,R_- := -R_+\,.$ 
\item[\rm{(3)}] $G(R_+,\lambda,z)=G(wR_+,w\lambda,wz)$. In particular, the hypergeometric function 
	$$ F(\lambda, z) =  F_k(\lambda,z)  := \frac{1}{|W|}\sum_{w\in W} G(R_+, \lambda, wz)$$
	does not depend on the choice of $R_+$.
\item[\rm{(4)}]  $G(R_+,\lambda,-z)=G(R_+, -w_0\lambda,w_0z)$, where $w_0$ is the longest element of $W.$
\item[\rm{(5)}] $F(\lambda,-z)=F(-\lambda,z)$.
\end{enumerate}
\end{proposition}

\begin{proof}
(1) This follows from the identities
$$w(\partial_\xi-\braket{\rho(R_+),\xi})w^{-1} = \partial_{w\xi}-\braket{\rho(wR_+),w\xi}$$
and
\begin{align*}
\sum\limits_{\alpha \in R_+} k_\alpha \braket{\alpha,\xi} w\frac{1-s_\alpha}{1-e^{-\alpha}}w^{-1}  
&=\sum\limits_{\beta \in wR_+} k_\beta \braket{\beta,\xi} \frac{1-s_\beta}{1-e^{-\beta}}.
\end{align*}

\noindent
(2)  Note that  $(\partial_\xi-\braket{\rho(R_+),\xi})f^- = -(\partial_\xi-\braket{\rho(-R_+),\xi} f)^-$. As $k_\alpha = k_{-\alpha},$  we also have 
\begin{align*}
\sum\limits_{\alpha \in R_+} k_\alpha \braket{\alpha,\xi} \frac{f^- -s_\alpha f^-}{1-e^{-\alpha}} 
&=-\Bigl(\sum\limits_{\beta \in -R_+} k_\beta \braket{\beta,\xi}\frac{1-s_\beta}{1-e^{-\beta}} f\Bigr)^-.
\end{align*}
Thus  $\,D_\xi(R_+)f^- = -(D_\xi(-R_+)f)^-$.

\noindent
(3) In view of part (1), the defining eigenvalue equation $D_\xi(R_+)f=\braket{\lambda,\xi} f$ is equivalent to
$$D_{w\xi}(wR_+)(wf)=wD_\xi(R_+)f=\braket{\lambda,\xi}(wf)=\braket{w\lambda,w\xi}(wf).$$
Hence $G(R_+,\lambda,w^{-1}z)=G(wR_+,w\lambda,z)$.

\noindent
(4) From part (2) we  observe that
$$D_\xi(R_-)G(\lambda,-\m) = -(D_\xi(R_+)G_k(\lambda,\m))^- = -\braket{\lambda,\xi}G(\lambda,-\m).$$
Therefore $G(R_+, \lambda,-z)=G(R_-,-\lambda,z)$. 
The longest element $w_0 \in W$ satisfies $\,w_0R_-=R_+$. Hence by part (3),
$$G(R_+,\lambda,-z)=G(w_0R_-,-w_0\lambda, w_0z)=G(R_+,-w_0\lambda,w_0 z).$$

\noindent
(5) This is clear from part (4).
\end{proof}

If $k \geq 0,$ then according to \cite{Sch08} and \cite{Opd95}, $G(\lambda, x) >0$ for all $\lambda, x\in \mathfrak a$ 
and 
\begin{equation}\label{estim_realkernel} |G(\lambda, x)| \leq G(\text{Re}\,\lambda, x) \leq \sqrt{|W|} \, e^{\max_{w\in W} \langle \text{Re}\,\lambda, wx\rangle}\quad \text{ for all }\,\lambda \in \mathfrak a_{\mathbb C},\, x\in \mathfrak a.
	\end{equation}

The following result generalizes the estimate of the Opdam-Cherednik kernel stated in \cite[Theorem 3.3]{RKV13}. Notice that the eigenvalue characterization of $G$ implies that $G(-\rho,\m) \equiv 1 $ for $\rho = \rho(R_+).$ 

\begin{proposition}\label{CherednikEstimateSum} Suppose that $k\geq 0.$ Then the 
Cherednik kernel $G=G_k(R_+, \m,\m)$ satisfies the following estimate for $x\in \mathfrak a$ and all $\lambda, \mu\in \mathfrak a_{\mathbb C}$:
$$\abs{G(\lambda+\mu,x)} \leq G(\text{Re}\,\mu\,,x)\m e^{\max_{w \in W} \braket{\text{Re}\,\lambda,wx}}.$$
Moreover, since $G(-\rho,\m) \equiv 1,$ we in particular have
$$\abs{G(\lambda-\rho,x)} \leq e^{\max_{w \in W} \braket{\text{Re}\,\lambda,wx}} \quad \te{ for all } \lambda \in \mathfrak{a}_\C, \, x \in \mathfrak{a}.$$
\end{proposition}

\begin{proof}
In \cite[Theorem 3.3]{RKV13} it was proven by a Phragmen-Lindel\"of argument  
that for all $\lambda \in \mathfrak a,\, \mu \in \overline{\mathfrak a_+}\, $ and $x\in \mathfrak a$, 
\begin{equation}\label{estim_RKV} G(\lambda + \mu, x) \,\leq \, G(\mu, x) \cdot e^{\max_{w \in W} \braket{\lambda,wx}}.\end{equation}
An inspection of the proof in loc.cit.  shows that it can be carried out in exactly the same way when $\overline{\mathfrak a_+}$ is replaced by an arbitrary closed Weyl chamber $C \subset \mathfrak a.$ 
Therefore estimate \eqref{estim_RKV} extends to all $\mu \in \mathfrak a$, and 
the claim follows from \eqref{estim_realkernel}. 
\end{proof}

\begin{remark} For root system $R=A_{n-1}$, items (3)--(5) of Proposition \ref{CherednikMinus} as well as Proposition \ref{CherednikEstimateSum} are easily checked to remain valid for the extensions of $G$ and $F$ as defined in \eqref{G_extended}.
\end{remark}

\medskip

\section{Laplace transform of the Opdam-Cherednik kernel}\label{Cherednik}

In this section, we return to   $A_{n-1}\,$ and resume the  notations from Sections 1--3. We shall extend the statements of  Theorem \ref{master} to the Opdam-Cherednik kernel and the hypergeometric function of type $A_{n-1}.$  Formulas \eqref{connection_kernelpolys} and \eqref{connection_hypergeopolys} suggest that it will be convenient to work with suitable modifications of the (extended) kernels $G$ and $F.$ According to our construction in Section \ref{Dunkl}, the extended Opdam-Cherednik kernel $G$ is in particular holomorphic on the set 
$$ \mathbb C^n \times (\mathbb R^n + i \Omega^\prime)\, \text{ with }\, 
\Omega^\prime =\{ x \in \mathbb R^n: \,|x_i| < \tfrac{\pi}{2}\, \text{ for all }\, i=1, \ldots, n\}.$$
The exponential mapping $\exp: z \mapsto e^z$  (understood componentwise) 
maps  $\mathbb R^n + i\Omega^\prime$ biholomorphically onto $\,H:=\{z\in 
\mathbb C^n: \text{Re}\, z >0\}$ with inverse $\log$. 
The modified kernels
\begin{align*}
\mathcal{G}(\lambda,z) &:= G(\lambda,\log(z)), \\
\mathcal{F}(\lambda,z) &:= F(\lambda,\log(z)) 
\end{align*}
are therefore holomorphic on $\mathbb C^n \times H.$ 
We call them the rational version of the Opdam-Cherednik kernel and the hypergeometric function, respectively. Obviously $\mathcal{F}$ is $\mathcal{S}_n$-invariant in each argument. 
We collect some properties of $\mathcal G$ and $\mathcal F$ which shall be needed lateron. 

\begin{lemma}\label{GF_Properties} 
\begin{enumerate}
\item[\rm{(1)}] For all $z \in H, \lambda \in \mathbb C^n $ and $\mu \in \C$, 
$$\Delta(z)^\mu \,\mathcal{G}(\lambda,z)=\mathcal{G}(\lambda+\underline{\mu},z) , \quad \Delta(z)^\mu \,\mathcal{F}(\lambda,z)=\mathcal{F}(\lambda+\underline{\mu},z).$$
\item[\rm{(2)}]  For all $z\in H$ and $\lambda\in \mathbb C^n$, 
$$\mathcal{G}(\lambda,\tfrac{1}{z})=\mathcal{G}(-\lambda^R,z^R), \quad \mathcal{F}(\lambda,\tfrac{1}{z}) = \mathcal{F}(-\lambda,z), $$
\item[\rm{(3)}]  	For all $x\in \mathbb R_+^n$ and all $\lambda, \mu\in \mathbb C^n$, 
$$ |\mathcal G(\lambda + \mu, x)| \,\leq \,\mathcal G(\text{Re}\,\lambda, x)\cdot\max_{\sigma\in \mathcal S_n} x^{\sigma(\text{Re}\, \mu)}.$$
In particular, for all $x \in \mathbb R_+^n$ and $\lambda \in \mathbb C^n$, $$ |\mathcal G(\lambda-\rho, x)| \,\leq \,\max_{\sigma\in \mathcal S_n} x^{\sigma(\text{Re}\, \lambda)}$$
and
$$|\mathcal{G}(\lambda,x)| \,\leq \, \sqrt{n!}\max_{\sigma\in \mathcal S_n} x^{\sigma(\text{Re}\, \lambda)}.$$
The same estimates hold for $\mathcal F.$
\item[\rm{(4)}]  For partitions $\lambda \in \Lambda_n^+$,
$$ \mathcal G(\lambda-\rho,z) = \frac{E_\lambda(z)}{E_\lambda(\underline 1)}\,.$$
\end{enumerate}
	\end{lemma}
	
	\begin{proof} Part (1) is clear from the definitions. Part (2) follows from Proposition \ref{CherednikMinus}, 
since the longest element $w_0\in \mathcal{S}_n$ acts by $w_0\lambda=\lambda^R=(\lambda_n,\ldots,\lambda_1).$ Part (3) is immediate from Proposition \ref{CherednikEstimateSum} (and the subsequent remark). Finally, part (4) follows from identities \eqref{connection_kernelpolys} and \eqref{lambda-rho-id}.
\end{proof}

The extension of Theorem \ref{master} will be carried out by analytic extension with respect to the spectral parameter, 
which is based on the following generalization of the classical Carlson theorem \cite[p.186]{Tit76}.

\begin{lemma}\label{Carlson}
Let $U \subseteq \mathbb C^n$ be an open neighborhood of $\{\text{Re}\,z \geq 0\} \subseteq \mathbb C^n$ and let $f:U \to \mathbb C$ be holomorphic. Put $\nrm{z}_1:=\sum_{i=1}^n \abs{z_i}$. If $f$ satisfies
$$(\ast) \quad f(z)=\mathcal{O}(e^{c\nrm{z}_1}) \, \te{ for some }c<\pi \te{ and } \, f|_{\Lambda_n^+} \equiv 0,$$
then $f\equiv 0$.
\end{lemma}

\begin{proof}
We proceed by induction on  $n$. The case $n=1$ is Carlson's classical theorem. To achieve step $n-1 \to n$, consider for fixed $\lambda \in \Lambda_n^+$ the holomorphic function
$$f_\lambda:U' \to \C, \xi \mapsto f(\xi+\lambda_1,\lambda_2,\ldots,\lambda_n)$$
where $U' \tm \C$ is a suitable neighborhood of $\{ \text{Re}\,\xi \,+\lambda_1\ge 0\} \tm \mathbb C$. Then $f_\lambda|_{\N_0}\equiv 0$ and 
$$f_\lambda(\xi)=\mathcal{O}(e^{c\abs{\xi}})$$ with  $c$ as in $(\ast).$ Therefore $f_\lambda$ vanishes identically. From this we conclude that for fixed $\xi \in \C$ with $\text{Re}\,\xi\ge 0$, the function
$$g_\xi:\widetilde{U} \to \C, \quad w \mapsto f(\xi,w)$$
vanishes on $\Lambda_{n-1}^{+}$ for some suitable neighborhood $\widetilde{U} \tm \C^{n-1}$ of $\{\text{Re}\,w \ge 0\}$. Moreover
$$g_\xi(w)=\mathcal{O}(e^{c\nrm{(\xi,w)}_1})=\mathcal{O}(e^{c\nrm{w}_1}),$$
and by the induction hypothesis we obtain that $g_\xi$ vanishes identically. As $\xi$ was arbitrary, we obtain $f \equiv 0$.
\end{proof}

We are now in the position to prove the following generalization of Theorem \ref{master}.

\bigskip

\begin{theorem}\label{MacdonaldCherednik}
Let $\mu \in \mathbb C$ with $\text{Re}\, \mu >\mu_0=k(n-1), \, \lambda \in \mathbb C^n$ with $\text{Re}\,\lambda\geq 0$ and $z\in H.$  Then 
\begin{enumerate}
\item[\rm{(1)}] 
$\displaystyle 
\int_{\R_+^n} E(-z,x)\,\mathcal{G}(\lambda,x)\,\Delta(x)^{\mu-\mu_0-1}\omega(x)dx \,= \,
\Gamma_n(\lambda+\rho+\underline{\mu}) \, \mathcal{G}(\lambda,\tfrac{1}{z})\,\Delta(z)^{-\mu}.$
\item[\rm{(2)}] $\displaystyle 
\int_{\R_+^n} J(-z,x)\,\mathcal{F}(\lambda,x)\,\Delta(x)^{\mu-\mu_0-1}\omega(x)dx \,= \,
\Gamma_n(\lambda+\rho+\underline{\mu}) \, \mathcal{F}(\lambda,\tfrac{1}{z})\,\Delta(z)^{-\mu}.$
\end{enumerate}
\end{theorem}

In view of Lemma \ref{GF_Properties}, the above Theorem can be equivalently reformulated as 
follows:

\begin{corollary}\label{Cherednik_reform} Suppose that  $\text{Re}\,\lambda \geq \underline{\mu_0}\,$. Then for all $z\in H,$
\begin{enumerate}
\item[\rm{(1)}]
$\displaystyle
\int_{\R_+^n} E(-z,x)\,\mathcal{G}(\lambda,x)\,\Delta(x)^{-\mu_0-1}\omega(x) dx \,=\,
  \Gamma_n(\lambda+\rho) \, \mathcal{G}(\lambda,\tfrac{1}{z});$
\item[\rm{(2)}]
$\displaystyle
\int_{\R_+^n} J(-z,x)\,\mathcal{F}(\lambda,x)\,\Delta(x)^{-\mu_0-1}\omega(x)dx \,
=\Gamma_n(\lambda+\rho) \,\mathcal{F}(\lambda,\tfrac{1}{z}). $
\end{enumerate}
\end{corollary}

The  second formula generalizes the Laplace transform identity \eqref{spherical_Laplace} for spherical functions on a symmetric cone.


\begin{proof}[Proof of Theorem \ref{MacdonaldCherednik}]
It suffices to check part (1).  By Carlson's theorem, we shall prove that 
\begin{equation}\label{MasterCherednik} 
\frac{1}{\Gamma_n(\lambda+\rho+\underline\mu)}\int_{\mathbb R_+^n} 
E(-z,x)\,\mathcal{G}(\lambda,x)\,\Delta(x)^{\mu-\mu_0-1}\omega(x)dx \,= \,
 \mathcal{G}(\lambda,\tfrac{1}{z})\,\Delta(z)^{-\mu}.\end{equation}
Note first that  \eqref{MasterCherednik} holds for all $\,\lambda \in \Lambda_n^+ -\rho\,,$ by Theorem \ref{master} and Lemma \ref{GF_Properties}(4). 
The right hand side of \eqref{MasterCherednik} is holomorphic in $(\lambda,z,\mu)$ on 
$\mathbb C^n \times H \times\mathbb C.$
Moreover, the  left hand side exists and is continuous on $\overline H \times H\times \{\text{Re}\, \mu>\mu_0\}$ and holomorphic on 
$H\times H \times \{\text{Re}\, \mu>\mu_0\}$. Indeed, suppose that $\text{Re}\,z \geq \underline s $ for some $s>0.$ Then by \cite{R20}, 
$$\abs{E(-z,x)}\leq E(-\text{Re}\,z,x)\,\leq \, e^{-\langle\underline s,x\rangle}.$$ 
Together with Lemma \ref{GF_Properties}  we obtain for  $x \in \R_+^n$ 
$$ \big|E(-z,x)\,\mathcal G(\lambda,x)\Delta(x)^{\mu-\mu_0-1}\big|\,\leq \, \sqrt{n!}
e^{-\langle\underline s, x\rangle}\Delta(x)^{\text{Re}\, \mu - \mu_0 -1} \max_{\sigma\in \mathcal S_n} x^{\sigma(\text{Re}\,\lambda)}.$$
Hence the integral on the left hand side of formula \eqref{MasterCherednik}  exists and is (by standard arguments) continuous respectively holomorphic 
as stated. It therefore suffices to check  \eqref{MasterCherednik} for $z \in \R^n$ with $z>\underline{1}$ and $\mu \in \mathbb R$ with  $\mu > \mu_0$. 
We want to apply Carlson's Theorem \ref{Carlson} with respect to $\lambda$. As $z>\underline 1,$ the right hand side of \eqref{MasterCherednik} is bounded in $\lambda$ according to Lemma \ref{GF_Properties}, and it remains to control the growth of the left hand side.  For $\lambda \in H,$
define $\,\eta(\lambda) := (\lceil \text{Re}\,\lambda_1) \rceil,\ldots,\lceil \text{Re}\,\lambda_n)\rceil)_+ \in \Lambda_n^+\,.$
 Then for arbitrary $x \in \R_+^n,$ 
$$\max\limits_{\sigma \in \mathcal{S}_n} x^{\sigma(\text{Re}\,\lambda)} \,\leq \,\max\limits_{\sigma \in \mathcal{S}_n} (\underline{1}+x)^{\sigma(\text{Re}\,\lambda)} \leq\, P_{\eta(\lambda)}(\underline{1}+x),$$
because the coefficients of $P_{\eta(\lambda)}$ in its monomial expansion are nonnegative. 
Now recall the binomial formula \eqref{binom} for the Jack polynomials as well as the identity 
$$\int_{\R_+^n} e^{-\langle \underline 1,x\rangle}P_\kappa(x)\Delta(x)^{\mu-\mu_0-1}\omega(x)  dx \, = \, 
\Gamma_n(\kappa + \underline \mu) P_\kappa(\underline 1)$$ 
from \cite[(6.18)]{Mac13} (c.f. also \cite[Lemma 5.1]{R20}).
We may therefore estimate
\begin{align*}
\int_{\R_+^n} &\abs{E(-z,x)\,\mathcal{G}(\lambda,x)\,\Delta(x)^{\mu-\mu_0-1}}\omega(x)dx \\ 
\quad &\le \int_{\R_+^n} e^{-\langle \underline 1,x\rangle }\,P_{\eta(\lambda)}(\underline{1}+x)\,\Delta(x)^{\mu-\mu_0-1}\omega(x) dx \\
  & =\, \sum\limits_{\kappa\subseteq \eta(\lambda)} \binom{\eta(\lambda)}{ \kappa} \int_{\R_+^n} e^{-\langle \underline 1,x\rangle}P_\kappa(x)\,\Delta(x)^{\mu-\mu_0-1}\omega(x)  dx \\
\quad &= \sum\limits_{\kappa\subseteq\eta(\lambda)}  \binom{\eta(\lambda)}{ \kappa} \,P_\kappa(\underline 1)   \, \Gamma_n(\kappa+\underline{\mu}).
\end{align*}
By monotonicity of the classical gamma function, 
$$\Gamma_n(\kappa+\underline{\mu}) \le \Gamma_n(\eta(\lambda)+\underline{\mu})\le \Gamma_n((\text{Re}\,\lambda)_+ +\underline 1+\underline{\mu}).$$
Moreover, by Remark \ref{Polys_eins}, 
$$ \sum\limits_{\kappa\subseteq\eta(\lambda)}\binom{\eta(\lambda)}{ \kappa} P_\kappa(\underline 1) \, = \, P_{\eta(\lambda)}(\underline 2) \, = \, 2^{|\eta(\lambda)|}P_{\eta(\lambda)}(\underline 1)\,  \leq \, 2^{\|\lambda\|_1} \cdot Q(\lambda)$$
with some polynomial $Q\in \mathcal P. $
 Therefore
 \begin{align*}
I_{z,\mu}(\lambda)&:= \,\abs{\frac{1}{\Gamma_n(\lambda+\rho+\underline\mu)}\int_{\mathbb R_+^n} 
E(-z,x)\,\mathcal{G}(\lambda,x)\,\Delta(x)^{\mu-\mu_0-1}\omega(x)dx} \\&\,\leq\, Q(\lambda)\cdot 
\frac{\Gamma_n((\text{Re}\,\lambda)_+ +1+\underline{\mu})}{\abs{\Gamma_n(\lambda+\rho+\underline{\mu})}}\,\cdot 2^{\|\lambda\|_1}.
\end{align*}
Choose $\sigma \in \mathcal S_n$ with $\, \sigma(\text{Re}\,\lambda)=(\text{Re}\,\lambda)_+$\,. 
Then 
\begin{align*}
&\frac{\Gamma_n((\text{Re}\,\lambda)_+ +1+\underline{\mu})}{\abs{\Gamma_n(\lambda+\rho+\underline{\mu})}} = \prod\limits_{j=1}^n \frac{\Gamma\bigl(((\text{Re}\,\lambda)_+)_j+\mu+1-k(j-1)\bigr)}{\abs{\Gamma\bigl(\lambda_j+\mu+\rho(k)_j-k(j-1)\bigr)}} \, = \, F_1(\lambda) \cdot F_2(\lambda)\end{align*}
with 
\begin{align*}  F_1(\lambda) &= \,\prod_{j=1}^n \frac{\Gamma\bigl(\text{Re}\,\sigma(\lambda)_j + \mu +1 -k(j-1)\bigr)}{\Gamma\bigl(\text{Re}\,\sigma(\lambda)_j + \mu + 1- \frac{k}{2}(n-1)\bigr)}\,, \\
F_2(\lambda) &= \,
 \prod_{j=1}^n\frac{\abs{ \lambda_j + \mu - \frac{k}{2}(n-1)}  \cdot \Gamma\bigl(\text{Re}\,\lambda_j +\mu+ 1-\frac{k}{2}(n-1)\bigr)}{\abs{\Gamma\bigl(\lambda_j+  \mu + 1- \frac{k}{2}(n-1)\bigr)}}\,.
 \end{align*}
 By Stirling's formula, $F_1(\lambda)$ is polynomially bounded, i.e. $F_1(\lambda)  = \mathcal O(e^{\epsilon \|\lambda\|_1})$ for arbitrary $\epsilon >0.$
 For $F_2$, we employ the estimate (\cite[Formula 5.6.7]{NIST})
$$\frac{\Gamma(x)}{\abs{\Gamma(x+iy)}}\le \sqrt{\cosh(\pi y)}=\mathcal{O}(e^{\tfrac{\pi}{2}|y|}), \quad x>\tfrac{1}{2},y \in \R,$$ 
which leads to 
$$ F_2(\lambda) 
\, =\, \mathcal{O}(e^{(\epsilon + \tfrac{\pi}{2})\|\lambda\|_1}) $$
with arbitrary $\epsilon >0.$ 
Putting things together, we obtain that $ I_{z, \mu}(\lambda)$ satisfies the growth condition of Carlson's Theorem \ref{Carlson}, which finishes the proof. 
\end{proof}

\section{Macdonald's hypergeometric series and their Laplace transform}\label{Hypergeometric}

 In the setting of symmetric cones, the Laplace 
transform establishes important identities between hypergeometric series. Analogous formulas were formally stated by Macdonald \cite{Mac13} for general Jack-hypergeometric series,  as consequences 
of his conjecture $(C).$  With Theorem \ref{master} at hand, we shall make these identities precise, and extend them to hypergeometric expansions in terms of non-symmetric Jack polynomials. 
We start with the appropriate normalization of the symmetric and non-symmetric Jack polynomials. 

\begin{lemma}\label{C_normalization}
\begin{enumerate}\itemsep=-1pt
\item[\rm{(i)}] There are numbers $c_\eta >0$ for $\eta \in \mathbb N_0^n$ such that the renormalized Jack polynomials $\,C_\lambda:= c_\lambda P_\lambda\,$ and $\,L_\eta := c_\eta E_\eta\,$ satisfy
$$\sum_{\lambda \in \Lambda_n^+:|\lambda|=m} C_\lambda(z) = \sum_{\eta \in \mathbb N_0^n:|\eta|=m} L_\eta(z)\, =\, (z_1+\ldots+z_n)^m \quad\text{for all } m \in \mathbb N_0;$$
\item[\rm{(ii)}] $\displaystyle C_\lambda = \sum\limits_{\eta\in \mathcal S_n \lambda}  L_{\eta} \,\, $ for all $\,\lambda \in \Lambda_n^+$.
\item[\rm{(iii)}] $\displaystyle c_\lambda \leq \,\frac{|\lambda|!}{\lambda !}\,$ for all $\,\lambda \in \Lambda_n^+$.
\end{enumerate}
\end{lemma}

\begin{proof} We may assume that $k>0.$ 
Part (i)  for the symmetric Jack polynomials is well-known (see e.g. \cite[(12.135)]{For10}), with
$$c_\lambda=\frac{|\lambda|!}{k^{|\lambda|}d_\lambda'}.$$ 
Here the  constants $d_\eta'$ for $\eta \in \mathbb N_0^n$ are given by
$$ d_\eta^\prime = \prod_{(i,j)\in \eta } \Bigl(\frac{1}{k}(\eta_i-j +1) + \ell(\eta,i,j)\Bigr) >0, $$
with the leg length $\,\ell(\eta,i,j)=\#\set{\ell>i \mid j \le \eta_\ell\le \eta_i}\, + \,\#\set{\ell<i \mid j\le  \eta_\ell+1 \le \eta_i}$. 
In particular, for each partition $\lambda\in \Lambda_n^+$ we have
$$c_{\lambda}=\frac{\abs{\lambda}!}{\prod\limits_{(i,j)\in \lambda} ((\lambda_i-j+1)+k\ell(\lambda,i,j))} \le \frac{\abs{\lambda}!}{\prod\limits_{1\le j \le \lambda_i} (\lambda_i-j+1)}= \frac{\abs{\lambda}!}{\lambda!},$$
which is part (iii).
From \cite[Proposition 12.6.1]{For10} it is further known that
\begin{equation}\label{SymAndNonSymJack}
P_\lambda=d_\lambda'\sum\limits_{\eta \in\mathcal{S}_n\lambda} \frac{1}{d_{\eta}'}E_{\eta}.
\end{equation}
Hence, we put $$c_{\eta}:=c_{\eta_+}\frac{d_{\eta_+}'}{d_\eta'}=\frac{|\eta|!}{k^{|\eta|}\,d_\eta'}$$ for $\eta\in \mathbb N_0^n$,  and part (i) for the non-symmetric Jack polynomials follows. Finally, part (ii) is immediate from the definition of $ c_\eta$ and relation \eqref{SymAndNonSymJack}. 
\end{proof}
  
On the space $\mathcal{P}_{\mathbb R} = \mathbb R[\mathbb R^n]$ of real polynomials on $\mathbb R^n$ there exists an $\mathcal{S}_n$-invariant inner product $[\m,\m] = [\m,\m]_k$ called the Dunkl pairing (cf. \cite{Dun91}), which is defined by
$$[p,q]:= (p(T)q)(0).$$
Here the Dunkl operators are again those of type $A_{n-1}$ with multiplicity $k.$ 
Polynomials with different homogeneous degree are orthogonal with respect to this pairing, and
$\,[T_\xi p,q]=[p,\braket{\m,\xi}q].$ This property and the invariance under the action of $\mathcal{S}_n$ show that the Cherednik operators $\mathcal{D}_j$ are symmetric with respect to the Dunkl pairing. In particular, the non-symmetric Jack polynomials $(E_\eta)_{\eta \in \N_0^n}$ form an orthogonal basis of $\mathcal{P}_{\mathbb R}$ with respect to $[\,.\,,\,.].$ More precisely,  their renormalizations $L_\eta= c_\eta E_\eta$ satisfy
\begin{equation}\label{L_normalization}
	[L_\eta, L_\kappa] = \,|\eta|!\, L_\eta(\underline 1) \cdot 
	\delta_{\eta,\kappa}
\end{equation}
which is obtained by combining  Proposition 3.18 and formula (2.4) of \cite{BF98}. 

\begin{lemma}\label{E_series} The Dunkl kernel of type $A_{n-1}$ with multiplicity $k\geq 0$ satisfies
$$ E(z,w) = \sum_{\eta\in \mathbb N_0^n} \frac{L_\eta(z)L_\eta(w)}{|\eta|!\, L_\eta(\underline  1)}.$$
The series converges locally uniformly on $\mathbb C^n\times \mathbb C^n.$ 
\end{lemma}

\begin{proof} This is immediate from \cite[Lemma 3.1]{R98} together with identity \eqref{L_normalization}. Alternatively, the stated expansion follows from \cite[Propos. 13.3.4]{For10}. 
\end{proof}
 
\begin{definition}
Consider the Jack polynomials $(L_\eta)_{\eta \in \N_0^n}$ and $(C_\lambda)_{\lambda \in \Lambda_n^+}$ of index $\alpha=\tfrac{1}{k}$, respectively (normalized as above). 
	Following \cite{Mac13}, \cite{Kan93} and \cite{BF98},    we define for indices $\mu \in \C^p$ and $\nu\in \C^q$ with $p, q \in \mathbb N_0$ the non-symmetric hypergeometric series
$${}_pK_q(\mu;\nu;z,w) \,:= \sum\limits_{\eta \in \N_0^n} \frac{[\mu_1]_{\eta_+}\cdots [\mu_p]_{\eta_+}}{[\nu_1]_{\eta_+}\cdots [\nu_q]_{\eta_+}} \,\frac{L_\eta(z)L_\eta(w)}{\abs{\eta}!\,L_\eta(\underline{1})}$$
as well as the symmetric 	hypergeometric series
	$${}_pF_q(\mu;\nu;z,w) \,:= \sum\limits_{\lambda \in \Lambda_n^+} \frac{[\mu_1]_{\lambda}\cdots [\mu_p]_{\lambda}}{[\nu_1]_{\lambda}\cdots [\nu_q]_{\lambda}} \,\frac{C_\lambda(z)C_\lambda(w)}{\abs{\lambda}!\,C_\lambda(\underline{1})}.$$
\end{definition}
More common in the literature are hypergeometric series in one variable, which  are obtained  as functions in $z$ by setting $w=\underline{1}$. 
For abbreviation, we write for $\lambda \in \Lambda_n^+$ 
$$[\mu]_\lambda:=  [\mu_1]_{\lambda}\cdots [\mu_p]_{\lambda}; \quad [\nu]_\lambda :=[\nu_1]_{\lambda}\cdots [\nu_q]_{\lambda}.$$
Note that for $p=0$ or $q=0,$ an empty product occurs.
For those  values of $k$ for which the  $C_\lambda= C_\lambda(\m\,; \frac{1}{k})  $ are the spherical polynomials of a symmetric cone, the convergence properties of $\,_pF_q$-hypergeometric series in one variable  are  well-known, see \cite{FK94, GR89}. For general $k>0,$ partial results on the domain of convergence of $\,_pF_q$ were obtained in \cite{Kan93}. For some values of $p$ and $q$, the nonsymmetric series $_pK_q$ were considered in \cite{BF98}. But to our knowledge, their convergence properties have not been studied so far.

\begin{lemma}\label{SymmetrizationHypergeometric}
The non-symmetric and symmetric hypergeometric functions are related by
$$\frac{1}{n!}\sum\limits_{\sigma \in \mathcal{S}_n} {}_pK_q(\mu;\nu;\sigma z,w) = {}_pF_q(\mu;\nu;z,w).$$
\end{lemma}

\begin{proof}
By identity \eqref{Symm_1} and Lemma \ref{C_normalization}  we have
\begin{align*}
\frac{1}{n!}\sum\limits_{\sigma \in \mathcal{S}_n} {}_pK_q(\mu;\nu;\sigma z,w) 
&= \sum\limits_{\eta \in \N_0^n} \frac{[\mu]_{\eta_+}}{[\nu]_{\eta_+}} \frac{1}{\abs{\eta}!} L_\eta(w)\frac{C_{\eta_+}(z)}{C_{\eta_+}(\underline{1})} \\
&= \sum\limits_{\lambda \in \Lambda_n^+} \frac{[\mu]_\lambda}{[\nu]_\lambda} \frac{1}{\abs{\lambda}!} \Bigl( \sum\limits_{\eta \in \mathcal{S}_n \lambda}L_{\eta}(w)\Bigr)\,\frac{C_{\lambda}(z)}{C_{\lambda}(\underline{1})} \\
&={}_pF_q(\mu;\nu;w,w).
\end{align*}
\end{proof}

\begin{theorem}\label{ConvergenceHypergeometricSeries}
Let $\mu \in \C^p$ and $\nu \in \C^q$ with $\nu_i \notin \set{0,k,\ldots,k(n-1)}-\N_0$ for all $i=1,\ldots,n$ (i.e. $[\nu]_\lambda \neq 0$ for all $\lambda \in \Lambda_n^+)$.
\begin{enumerate}\itemsep=+1pt
\item[\rm{(1)}] If $p\leq q$,  the series ${}_pK_q(\mu;\nu;\m,\m)$ and ${}_pF_q(\mu;\nu;\m,\m)$ are entire functions.
\item[\rm{(2)}] If $p=q+1$,  the  series ${}_pK_q(\mu;\nu;\m,\m)$ and ${}_pF_q(\mu;\nu;\m,\m)$ are holomorphic on the domain $\{(z,w) \in \mathbb C^{n}\times \mathbb C^n: \nrm{z}_\infty\nrm{w}_\infty <1\}$.
\end{enumerate}
Moreover, the hypergeometric series are holomorphic in the parameters $(\mu,\nu)$ on the domain 
$$\set{(\mu,\nu)\in \C^p\times \C^q \mid \nu_i \notin \set{0,k,\ldots,k(n-1)}-\N_0 \te{ for all }i=1,\ldots,n}.$$
\end{theorem}

\begin{proof} It suffices to verify the statements for $\,_pK_q$. 
 From Lemma \ref{NonSymmetricJackProperties}(4)  we have 
$\, \abs{L_\eta(z)} \leq \, L_\eta(|z|) \leq L_\eta(\underline{1})\nrm{z}_\infty^{\abs{\eta}}\,$ and therefore
\begin{align*}\label{K_abschaetz} S(\mu,\nu;z,w)&:= \sum_{\eta\in \mathbb N_0^n} \abs{\frac{[\mu]_{\eta_+}}{[\nu]_{\eta_+}}}\cdot
\abs{\frac{L_\eta(z)L_\eta(w)}{|\eta|!\, L_\eta(\underline 1)}}
\, \leq \, \sum_{\eta\in \mathbb N_0^n} \abs{\frac{[\mu]_{\eta_+}}{[\nu]_{\eta_+}}}\cdot\frac{\|z\|_\infty^{|\eta|} \|w\|_\infty^{|\eta|}  }{|\eta|!}\, L_\eta(\underline 1)\\
& =\, \sum_{\lambda\in \Lambda_n^+} \abs{\frac{[\mu]_\lambda}{[\nu]_\lambda}}\cdot
\frac{\|z\|_\infty^{|\lambda|}\|w\|_\infty^{|\lambda|} }{|\lambda|!} \,C_\lambda(\underline 1),
\end{align*}
where for the last identity, Lemma \ref{C_normalization}(ii) was used.  
From Lemmata \ref{Polys_eins} and \ref{C_normalization} we know that 
\begin{equation}\label{estim_Jack_1} C_\lambda(\underline 1) \,
= \,c_\lambda P_\lambda(\underline 1)\, \leq \,\frac{|\lambda|!}{\lambda!} Q(\lambda)\end{equation}
with some polynomial $Q\in \mathcal P.$ 
Therefore, we can find to each $\epsilon>1$ a constant $C_\epsilon>0$ such that $Q(\lambda)\le C_\epsilon\,\epsilon^{\abs{\lambda}}$. This gives
\begin{equation}\label{S_estimate} S(\mu,\nu;z,w) \leq \,C_\epsilon \sum_{\lambda\in \Lambda_n^+} \abs{\frac{[\mu]_\lambda}{[\nu]_\lambda}}\cdot \frac{\bigl(\epsilon \|z\|_\infty\|w\|_\infty\bigr)^{|\lambda|}}{\lambda!}.\end{equation}
To prove part (1), consider the case $p\leq q.$ In this case,  the quotient
$$  \frac{[\mu]_\lambda}{[\nu]_\lambda} \,=\, 
     \frac{\prod_{i=1}^p [\mu_i]_\lambda}{\prod_{i=1}^q [\nu_i]_\lambda}$$
is of polynomial growth in $\lambda$. To see this, write
$$\frac{[\mu]_\lambda}{[\nu]_\lambda}=\prod\limits_{i=1}^p \prod\limits_{j=1}^n \frac{\Gamma(\nu_i-k(j-1))}{\Gamma(\mu_i-k(j-1))} \frac{\Gamma(\mu_i+\lambda_j-k(j-1))}{\Gamma(\nu_i+\lambda_j-k(j-1))}.$$
By Stirling's formula we have, locally uniformly in $\mu$ and $\nu$,
$$\frac{\Gamma(\mu_i+\lambda_j-k(j-1))}{\Gamma(\nu_i+\lambda_j-k(j-1))} \sim (\nu_i+\lambda_j-k(j-1))^{\nu_i-\mu_i} \te{ for } \lambda_j \to \infty.$$
Moreover, $|[\nu_i]_\lambda| \geq 1 $ for large $\lambda$.
Thus, for each $\epsilon >1$ there are constant $D>0$ and a compact neighborhood $K\tm \C^p \times \C^q$ of $(\mu,\nu)$, such that
$$\abs{\frac{[\mu]_\lambda}{[\nu]_\lambda}}\le D\epsilon^{\abs{\lambda}} \te{ for all } (\mu,\nu) \in K.$$ 
Hence, for each $\epsilon >1$, we find a constant $C_\epsilon>0$ such that 
\begin{align*} S(\mu,\nu;z,w) \, &\leq\,  C_\epsilon \sum_{\lambda\in \Lambda_n^+} \frac{\bigl(\epsilon \|z\|_\infty\|w\|_\infty\bigr)^{|\lambda|}}{\lambda!}\,\leq \,
C_\epsilon \sum_{\lambda\in \mathbb N_0^n} \frac{\bigl(\epsilon \|z\|_\infty\|w\|_\infty\bigr)^{|\lambda|}}{\lambda!} \notag \\ &\leq \, C_\epsilon\, e^{n\epsilon \|z\|_\infty\|w\|_\infty}.\end{align*}
 Therefore the $\, _pK_q$-series is converges locally uniformly on $\mathbb C^n\times \mathbb C^n$ and also locally uniformly on the stated domain of parameters $\mu$ and $\nu$, which proves part (1).

For part (2), observe that for $p=q+1,$ we have
$$ \frac{[\mu]_\lambda}{[\nu]_\lambda} \,=\, 
     \frac{\prod_{i=1}^q [\mu_i]_\lambda}{\prod_{i=1}^q [\nu_i]_\lambda}\cdot [\mu_{p}]_\lambda\,.$$
  As in part (1), the first factor is of polynomial growth.    Moreover,
  $$ \frac{[\mu_p]_\lambda}{\lambda!} \, =\, \prod_{j=1}^n \frac{ \Gamma(\mu_p -k(j-1)+ \lambda_j)}{\Gamma(\lambda_j+1)\, \Gamma(\mu_p-k(j-1))},$$
   which is of polynomial growth as well. Starting from estimate \eqref{S_estimate}, we therefore obtain that for each $\epsilon >1,$ there is a constant $ C_\epsilon>0$ with
   
$$S(\mu,\nu;z,w) \, \leq\, C_\epsilon \sum_{\lambda\in \Lambda_n^+} \bigl(\epsilon\|z\|_\infty\|w\|_\infty\bigr)^{|\lambda|}   \, \leq \, C_\epsilon \,\frac{1}{(1- \epsilon\|z\|_\infty\|w\|_\infty)^n} \,.$$
     This yields the claim. 
   \end{proof}
   
   Note that part (2) of this theorem improves, in the case $w=1$, the results of \cite{Kan93}. 
   
   \begin{remark}\label{Dunkl_hyper_connect} For $p=q=0$, one gets the Dunkl kernel and Bessel function of type $A_{n-1},$  respectively. Indeed, Lemma \ref{E_series} just says that 
$$ E(z,w)={}_0K_0(z,w),$$
and symmetrization yields
$$ J(z,w) = {}_0F_0(z,w),$$
which was already  noted in \cite{BF98}. 
\end{remark}
    
\begin{remark} The proof of Theorem \ref{ConvergenceHypergeometricSeries} shows that for $ p \leq q$ and arbitrary $\epsilon >1$ there is a constant $C_\epsilon >0$ such that 
\begin{equation}\label{HypGeoGrowthpq} \big\vert {}_pK_q(\mu;\nu;z,w) \big\vert \leq \,S(\mu,\nu;z,w) \,\leq \,
 C_\epsilon\, e^{n\epsilon \|z\|_\infty\|w\|_\infty}.\end{equation}
 Taking a closer look at the above proof, we see that for $p<q$ this estimate can be improved. 
 Indeed, consider the quotient 
$$\frac{[\mu]_\lambda}{[\nu]_\lambda}= \prod\limits_{j=1}^p \frac{[\mu_j]_\lambda}{[\nu_j]_\lambda} \m \prod\limits_{j={p+1}}^q\frac{1}{[\nu_j]_\lambda}.$$
By Stirling's formula, the first factor is of polynomial growth, and thus of order $\mathcal O ({\epsilon_1}^{\abs{\lambda}})$ for arbitrary $\epsilon_1>1,$ while the second factor is of order $\mathcal O(\epsilon_2^{-\abs{\lambda}})$ for arbitrary $\epsilon_2>1$. Under the assumption $p<q$ we therefore obtain the estimate
\begin{equation}\label{HypGeoGrowthp<q}
\big\vert {}_pK_q(\mu;\nu;z,w) \big\vert \leq \,S(\mu,\nu;z,w) \,\leq \,
 C_\epsilon\, e^{\epsilon \|z\|_\infty\|w\|_\infty}.
\end{equation}
for  arbitrary $\epsilon>0,$ with some constant $C_\epsilon>0$.
\end{remark}

The domain of convergence of the hypergeometric series $_pK_q$ and $_pF_q$ and their growth estimates \eqref{HypGeoGrowthpq}, \eqref{HypGeoGrowthp<q} are important to obtain from the Laplace transform identities for Jack polynomials in Theorem \ref{master}  similar Laplace transform identities for the hypergeometric series. 

\begin{theorem}\label{LaplaceTrafoHypgeoSeries}
Let $\mu \in \C^p$, $\nu \in \C^q$ with $\nu_i \notin \set{0,k,\ldots,k(n-1)}-\N_0$ for all $i=1,\ldots,n$ and let $\mu' \in \C$ with $\text{Re} \,\mu'>\mu_0$.
\begin{enumerate}
\item[\rm{(1)}] If $p<q$, then for all $z,w \in \C^n$ with $\text{Re}\,z>0,$
\begin{align*}
\int_{\R_+^n}E(-z,x)\, {}_pK_q(\mu;\nu;w,x) &\Delta(x)^{\mu'-\mu_0-1} \omega(x) dx \\
&=\Gamma_n(\mu')\Delta(z)^{-\mu'}\, {}_{p+1}K_q((\mu',\mu);\nu;w,\tfrac{1}{z}).
\end{align*}
\item[\rm{(2)}] If $p=q$, then part (1) is valid under the condition $\nrm{w}_\infty\m\nrm{\tfrac{1}{\text{Re}\,z}}_\infty<\tfrac{1}{n}$.
\end{enumerate}
Moreover, both parts remain valid if ${}_pK_q$ is replaced by ${}_pF_q$. 
\end{theorem} 

\begin{proof} (1)
By expanding  ${}_pK_q$ into its defining series, this is immediate from the Laplace transform identity of Theorem \ref{master} by interchanging the order of summation and integration. We have to justify this interchange. Choose $\epsilon>0$ such that  $\nrm{w}_\infty\m\nrm{1/\text{Re}\,z}_\infty < \tfrac{1}{\epsilon}$. Under these conditions the estimates \eqref{DunklKernelEstimate}, \eqref{HypGeoGrowthpq} and \eqref{HypGeoGrowthp<q} show that
$$\abs{E(-x,z)}S(\mu,\nu;w,x) \le C_\epsilon\, e^{-d_\epsilon \|x\|_\infty} $$
with $d_\epsilon = \min\limits_{i=1\ldots n} \text{Re}\, z_i -\epsilon\|w\|_\infty \, >0.$
Hence, we can apply the dominated convergence theorem to justify the interchange of summation and integration, so that part (1) is proven since $\epsilon>0$ was  chosen arbitrarily. 
Part (2) is obtained in 
the same way, by choosing $\epsilon >1$ such that $\nrm{w}_\infty\m\nrm{1/\text{Re}\,z}_\infty < \tfrac{1}{n\epsilon}$.
By symmetrization we  get  the same identities for ${}_pF_q$.
\end{proof}

We continue with an integral representation which was already observed in \cite[p.39]{Mac13} for the symmetric case, i.e.  for ${}_1F_0$, but only at a formal level and without any statement on convergence.

\begin{corollary}\label{6.8.}
Let $\mu \in \C$ with $\text{Re}\,\mu\,>\mu_0$. Then the hypergeometric series ${}_1K_0(\mu;-z,w)$ has an analytic continuation to $\,D:=\set{\te{Re}\, z>0}\times \set{\te{Re}\, w >0}$ which is 
given by
$${}_1K_0(\mu;-z,w)=\frac{\Delta(z)^{-\mu}}{\Gamma_n(\mu)}\int_{\R_+^n}E(-\tfrac{1}{z},x)E(-w,x)\Delta(x)^{\mu-\mu_0-1}\omega(x) dx.$$
By symmetrization, the same formula is valid if one replaces ${}_1K_0$ by ${}_1F_0$ and the Dunkl kernel by the Bessel function.
\end{corollary}

\begin{proof}
Recall that 
$\,E(z,w)={}_0K_0(z,w).$
Then, by Theorem \ref{LaplaceTrafoHypgeoSeries}, the stated integral formula holds on 
a suitable open subset of $D. $
Moreover,  estimate \eqref{DunklKernelEstimate} for the Dunkl kernel shows that the integral exists
and defines a holomorphic function on $D$ by standard theorems on holomorphic parameter integrals. Hence, analytic continuation finishes the proof.
\end{proof}

The following proposition is a generalization of the Euler integral for hypergeometric functions on symmetric cones ( \cite[Proposition XV.1.4]{FK94}) and can be found as a formal statement in \cite[formula (6.21)]{Mac13}. It will be obtained from Kadell's \cite{Kad97} generalization of the Selberg integral,
\begin{equation}\label{Kadell}
\int_{[0,1]^n} \frac{C_\lambda(x)}{C_\lambda(\underline{1})} \Delta(x)^{\mu-\mu_0-1}\Delta(\underline{1}-x)^{\nu-\mu_0-1} \omega(x) d x = \frac{\Gamma_n(\mu)\, \Gamma_n(\nu)}{\Gamma_n(\mu+\nu)} \frac{[\mu]_\lambda}{[\mu+\nu]_\lambda}
\end{equation}
for all $\lambda \in \Lambda_+^n$ and $\mu,\nu \in \C$ with $\text{Re} \, \mu, \text{Re} \, \nu > \mu_0$.

\begin{proposition}\label{Euler}
Consider $p \le q+1$ and $\mu',\nu' \in \C$ with $Re \, \mu', Re (\nu'-\mu')>\mu_0$. Moreover, let $\mu \in \C^p$ and $\nu \in \C^q$ with $\nu_i \notin \set{0,k,\ldots,k(n-1)}-\N_0$ for all $i=1,\ldots,n$. Then for arbitrary $w \in \C^n$ with the additional condition $\nrm{w}_\infty <1$ in the case $p=q+1,$ one has 
\begin{align*}
\int_{[0,1]^n} {}_pF_q(\mu;\nu;w,x)&\Delta(x)^{\mu'-\mu_0-1}\Delta(\underline{1}-x)^{\nu'-\mu'-\mu_0-1} \omega(x) d x \\
&= \frac{\Gamma_n(\mu) \, \Gamma_n(\nu-\mu)}{\Gamma_n(\nu)}{}_{p+1}F_{q+1}((\mu',\mu);(\nu',\nu); w,\underline{1}).
\end{align*}
\end{proposition}

\begin{proof}
This is immediate from Kadell's integral \eqref{Kadell} after expanding ${}_pF_q$ into its defining series and changing the order of integration and summation. The latter is justified since the series ${}_pF_q(\mu;\nu;w;\m)$ is absolutely bounded on $[0,1]^n$ by the estimates in the proof of Theorem \ref{ConvergenceHypergeometricSeries} for $w \in \C^n$ and $\nrm{w}_\infty < 1$ if $p=q+1$.
\end{proof}

The following theorem generalizes Proposition XV.1.2. of \cite{FK94} for hypergeometric series on symmetric cones. 

\begin{theorem}
The Jack polynomials and the hypergeometric series have the following properties under the action of the Dunkl operator $\Delta(T)$ associated to the polynomial $\Delta$.
\begin{enumerate}\itemsep=+1pt
\item[\rm{(1)}] 
 $\Delta(T)E_\eta=c_\eta E_{\eta-\underline{1}}\,$ with some constant $c_\eta \in \mathbb R.$ 
 Moreover,  $c_\eta=0$ if $\eta_i=0$ for some $i\in\{1,\ldots,n\}.$

\item[\rm{(2)}] $\eta \mapsto c_\eta$ is $\mathcal{S}_n$-invariant.

\item[\rm{(3)}]  $\Delta(T)L_\eta= d_\eta L_{\eta-\underline{1}}$ and $\Delta(T)C_\lambda=d_\lambda C_{\lambda-\underline{1}}$, where 
$$d_\eta=\begin{cases}
\displaystyle \frac{\abs{\eta}!}{\abs{\eta-\underline{1}}!}\, & \te{ if } \eta_i\neq 0 \te{ for all } i=1,\ldots,n \\
\,0\, & \te{ otherwise.} 
\end{cases}.$$
\item[\rm{(4)}]  If $p\leq q+1,$ then $$ \Delta(T)\, {}_pK_q(\mu;\nu;w,\m) = \frac{[\mu]_{\underline{1}}}{[\nu]_{\underline{1}}}\Delta(w)\, {}_pK_q(\mu+\underline{1},\nu+\underline{1};w,\m) $$
 for all  $w \in \C^n,$ with the understanding that $[\mu]_{\underline 1} = 1$ if $p=0$ and $[\nu]_{\underline 1} = 1$ if $q=0$. The same is valid for ${}_pK_q$ instead of ${}_pF_q$.
\end{enumerate}
\end{theorem}

\begin{proof} (1)  From the properties of the Dunkl pairing together with Lemma \ref{NonSymmetricJackProperties}\,(1) we can conclude that for compositions $\eta,\kappa \in \N_0^n$,
$$\; [\Delta(T)E_\eta,E_\kappa]=[E_\eta,\Delta E_\kappa]=[E_\eta,E_{\kappa+\underline{1}}]=\begin{cases}
\,0\, & \te{ if } \eta \neq \kappa + \underline{1}\,; \\
[E_\eta,E_\eta]>0\, & \te{ if } \eta= \kappa + \underline{1}\,.
\end{cases}$$
Hence, $\Delta(T)E_\eta$ must be a scalar multiple of $E_{\eta-\underline{1}}$ if $\eta_i \ge 1$ for all $i=1,\ldots,n$ and vanishes otherwise.

\smallskip
(2) Denote again by $\overline \eta_i$ the eigenvalue of $E_\eta$ under the Cherednik operator $\mathcal D_i\,.$   It suffices to show that $c_\eta = c_{s_i\eta}$ if $\eta_i < \eta_{i+1}.$  
Due to \cite[Proposition 12.2.1]{For10},  we then have
\begin{equation}\label{E_si} E_{s_i\eta} = d_i^\eta E_\eta + s_iE_\eta\end{equation}  
with the constant
$$d_i^\eta = \frac{k}{\overline{\eta}_{i+1}-\overline{\eta}_i}\,.$$
It is immediate that $\overline{\eta+\underline{1}}=\overline{\eta} + \underline{1}$ and therefore  $d_i^{\eta+\underline{1}}=d_i^{\eta}\,.$ 
Applying $\Delta(T)$ to equation \eqref{E_si}, using $d_i^{\eta+\underline{1}}=d_i^\eta$ and the $\mathcal{S}_n$-equivariance of the Dunkl operators, we obtain from part (1) that $c_{s_i\eta}=c_\eta\,.$ 

\smallskip
(3) As $L_\eta$ is a renormalization of $E_\eta$, there is a constant $d_\eta$ such that $\Delta(T)L_\eta = d_\eta L_{\eta-\underline 1},$  and $d_\eta = 0$ if $\eta_i =0$ for some $i.$ Since $\, \eta \mapsto \frac{\abs{\eta}!}{\abs{\eta-\underline{1}}!}$ is $\mathcal{S}_n$-invariant and $\,C_\lambda=\sum_{\eta \in \mathcal{S}_n\lambda } L_{\eta}\,,$ it suffices to verify the stated value of $d_\eta$ with $|\eta|\geq n.$ Recall that $T_i$ acts as $\partial_i\,$ on symmetric polynomials. We therefore conclude that for $m \in \N_0$ with $m\geq n,$ 
\begin{align*}
\quad m\cdots  (m-n+1)&\!\! \sum\limits_{\substack{\eta \in \N_0^n: \\ \abs{\eta}=m-n}} L_\eta(x) \,=\, m\cdots (m-n+1)(x_1+\ldots+x_n)^{m-n} \\
&= \Delta(T)(x_1+\ldots+x_n)^m = \sum\limits_{\substack{\eta \in \N_0^n : \\ \abs{\eta}=m}} \Delta(T)L_\eta \,=\,
\sum\limits_{\substack{\eta \in \N_0^n :\\ \abs{\eta}=m}} d_\eta L_{\eta-\underline{1}}.
\end{align*}
Equating the coefficients proves the stated formula for $d_\eta$.

\smallskip

(4) This is an immediate consequence of part (3) by expanding the hypergeometric series. One has to perform an index shift $\eta \mapsto \eta+\underline{1}$ after applying $\Delta(T)$ and the identity of part (3) together with $$\frac{L_\eta(w)}{L_\eta(\underline{1})}=\Delta(w)\frac{L_{\eta-\underline{1}}(w)}{L_{\eta-\underline{1}}(w)}$$ and $[\theta]_{\eta_+}=[\theta+1]_{\eta_+-\underline{1}}[\theta]_{\underline{1}}$.
\end{proof}

\section{A Post-Widder inversion formula for the Dunkl-Laplace transform}\label{Post-Widder}

We consider again the Dunkl setting of type $A_{n-1}$ with multiplicity $k\geq 0$ and keep our previous notations. 
As shown in  \cite{R20}, the Dunkl-Laplace transform 
$$ \mathcal Lf(z) = \int_{\mathbb R_+^n} f(x) E(-z,x) \omega(x)dx $$
satisfies the following Cauchy inversion theorem: 
Let $f\in L_{loc}^1(\mathbb R_+^n)$ such that $\mathcal Lf(\underline s)$ exists for some $s\in \mathbb R$ (then $\mathcal Lf(z)$ also exists for all $z\in \mathbb C^n$ with $\text{Re} \,z = \underline s\,$). Assume further that $ y\mapsto \mathcal Lf(\underline s+ iy) \in L^1(\mathbb R^n, \omega).$ Then $f$ has a continuous representative $f_0$, and 
$$ \frac{(-i)^n}{c_k^2}\int_{\text{Re}\, z = \underline s} \mathcal Lf(z) E(x,z)\omega(z)dz = \,\begin{cases} f_0(x) &\text{ for } x\in \mathbb R_+^n;\\
     \,0 & \text{ otherwise},
 \end{cases}$$
with the constant $\,c_k = \int_{\mathbb R^n} e^{-|x|^2/2} \omega(x)dx.$
For the classical Laplace transform
$$ Lf (z) = \int_0^\infty f(x) e^{-zx}dx, \quad f\in L_{loc}^1(\mathbb R_+),$$  
a further well-known inversion theorem is the Post-Widder inversion formula (see e.g. \cite{ABHN01}):
Assume that $f\in L_{loc}^1(\mathbb R_+)$ has a finite abscissa of convergence 
and is continuous in $\xi\in \mathbb R_+$. Then 
$$ f(\xi) = \lim_{\nu\to \infty} \frac{(-1)^\nu}{\nu!} \Bigl(\frac{\nu}{\xi}\Bigr)^{\nu+1}(Lf)^{(\nu)}\Bigl(\frac{\nu}{\xi}\Bigr).$$

In this section, we prove a Post-Widder inversion formula for the Dunkl-Laplace transform,  which is the counterpart to a result of Faraut and Gindikin  \cite{FG90} in 
the setting of symmetric cones.

\begin{theorem}[Post-Widder inversion formula for $\mathcal L$]\label{Th_Post-Widder} Let $f: \mathbb R_+^n\to \mathbb C$ be measurable and bounded, and suppose that $f$ is continuous at $\xi \in \mathbb R_+^n\,.$ Then
$$ f(\xi) = \lim_{\nu\to\infty} \frac{(-1)^{n\nu}}{\Gamma_n(\nu+\mu_0+1)} \Delta\Bigl( \frac{\nu}{\xi}\Bigr)^{\nu+\mu_0+1} \bigl(\Delta(T)^\nu 
(\mathcal Lf)\bigr)\Bigl(\frac{\nu}{\xi}\Bigr). $$

\end{theorem}
The idea of proof for this theorem is similar to \cite{FG90}. It was  elaborated to some extent by Frederik Hoppe in his master thesis \cite{Hop20}, which was supervised by the second author of this paper. 
A fundamental ingredient is Levy's continuity theorem for the Dunkl transform. Let us recall this for the reader's convenience.  Denote by $M_b^+(\mathbb R^n)$ the space of positive bounded Borel measures on $\mathbb R^n$. The Dunkl transform of $\mu\in M_b^+(\mathbb R^n)$ (associated with $A_{n-1}$ and multiplicity $k$) is given by 
$$ \widehat \mu(\xi) = \widehat \mu^{\,k}(\xi)= \,\int_{\mathbb R^n} E(-i\xi,x) d\mu(x), \quad \xi \in \mathbb R^n.$$
Note that $\widehat \mu\in C_b(\mathbb R^n), $ since  $|E(-i\xi,x)|\leq 1$ for all $\xi,x\in \mathbb R^n$.  The Dunkl transform is injective on $M_b^+(\mathbb R^n), $ see \cite{RV98}. 
The following is the essential part of Levy's continuity theorem for the Dunkl transform. 

\begin{lemma}[\cite{RV98}]\label{Levy} Let $(\mu_\nu)_{\nu\in \mathbb N}\subseteq M_b^+(\mathbb R^n)$  
such that the sequence    $(\widehat \mu_{\nu})_{\nu\in \mathbb N}$ converges pointwise to a function $\varphi:\mathbb R^n\to \mathbb C$ which is continuous at $0$. 
Then there exists a unique $\mu\in M_b^+(\mathbb R^n)$ with $\,\widehat \mu_{\nu} = \varphi$,
and $(\mu_\nu)_{\nu\in \mathbb N}$ converges to $\mu$ weakly. \end{lemma}

\smallskip

\begin{proof}[Proof of Theorem \ref{Th_Post-Widder}]

We consider on $\mathbb R_+^n$ the functions
$$ h_\nu (x):= E\bigl(-\tfrac{\nu}{\xi},x\bigr) \,\Delta(x)^\nu, \quad \nu\in \mathbb N.$$ 
By estimate \eqref{DunklKernelEstimate}, the Laplace transform
$$  \mathcal L h_\nu(z) =\int_{\mathbb R_+^n} E(-z,x) E\bigl(-\tfrac{\nu}{\xi},x\bigr) \Delta(x)^\nu \omega(x)dx $$ exists for all $z\in \mathbb C^n$ with $\text{Re}\, z\geq 0.$
For such $z$, put $\nu(z) := \max_i\lceil\|z\|_\infty\,\xi_i\rceil\in \mathbb N. $ Then 
for $\nu >\nu(z),$ we calculate 
\begin{align*} 
\mathcal L h_\nu(z) \,&=\,\int_{\mathbb R_+^n} \Bigl(\sum_{\eta\in \mathbb N_0^n} 
\frac{L_\eta(-z)L_\eta(x)}{|\eta|!\, L_\eta(\underline 1)} \Bigr) E\bigl(-\tfrac{\nu}{ \xi},x\bigr)\Delta(x)^\nu \omega(x)dx \\
&=\, \sum_{\eta\in \mathbb N_0^n} 
\frac{L_\eta(-z)}{|\eta|!\, L_\eta(\underline 1)} \int_{\mathbb R_+^n} L_\eta(x) E\bigl(-\tfrac{\nu}{\xi},x\bigr) \Delta(x)^\nu \omega(x)dx \\
&=\, \sum_{\eta\in \mathbb N_0^n} 
\frac{L_\eta(-z)}{|\eta|!\, L_\eta(\underline 1)}\, \Gamma_n(\eta_+ + \nu+\mu_0 +1)\, L_\eta\bigl(\tfrac{\xi}{\nu}\bigr) \,\Delta\bigl(\tfrac{\nu}{\xi}\bigr)^{-\nu-\mu_0-1}.
\end{align*}
Here the interchange of the sum and the integral is justified by the dominated convergence theorem, because $\, |L_\eta(-z)| \leq L_\eta(\|z\|_\infty \cdot\underline 1)$ and therefore
\begin{align*} E\bigl(-\tfrac{\nu}{\xi}, x\bigr) \sum_{\eta\in \mathbb N_0^n} 
\frac{|L_\eta(-z)L_\eta(x)|}{|\eta|!\, L_\eta(\underline 1)}\, &\leq E\bigl(-\tfrac{\nu}{\xi}, x\bigr)E(\|z\|_\infty\cdot\underline 1, x) \\
& = \,  E\bigl(-\tfrac{\nu}{\xi} + \|z\|_\infty\cdot \underline 1, x\bigr). 
\end{align*}
This decays exponentially on $\mathbb R_+^n$, since  $-\nu/\xi + \|z\|_\infty\cdot\underline  1 < 0$ by our assumption on $\nu.$ Thus for $\nu \geq \nu(z),$ 
\begin{equation}\label{c_eta-Formel} f_\nu(z):= \frac{\Delta\bigl(\frac{\nu}{\xi}\bigr)^{\nu+\mu_0+1}}{\Gamma_n(\nu+\mu_0+1)} \,\mathcal Lh_\nu(z) \,=\, 
 \sum_{\eta\in \mathbb N_0^n} c_\nu(\eta) \cdot \frac{ L_\eta(-z)L_\eta(\xi)}{|\eta|!\,L_\eta(\underline 1)}\end{equation}
with the coefficients
$$ c_\nu(\eta) = \frac{[\nu+\mu_0+1]_{\eta_+} }{\nu^{|\eta|}}\, =\, \prod_{j=1}^n \Bigl(1+ \frac{1+ k(n-j)}{\nu}\Bigr)_{\lambda_j}, \quad \lambda = \eta_+\,.$$
They satisfy $$ \lim_{\nu\to \infty} c_\nu(\eta) = 1 \quad \text{ for fixed } \,\eta,$$ and 
it follows that 
\begin{equation}\label{f_limit}  \lim_{\nu\to \infty} f_\nu(z)\,=\, 
 \sum_{\eta\in \mathbb N_0^n} \frac{ L_\eta(-z)L_\eta(\xi)}{|\eta|!\,L_\eta(\underline 1)} \, =\, E(-z, \xi).\end{equation}
 We still have to justify that the limit $\nu\to \infty$ may be taken inside the sum in \eqref{c_eta-Formel}.  For this, note that $\nu \mapsto c_\nu(\eta)$ is monotonically decreasing. Hence for $\nu\geq \nu(z),$ 
 the series on  the right-hand side of \eqref{c_eta-Formel}  
 is dominated by the convergent series
 $$ \sum_{\eta\in \mathbb N_0^n} c_{\nu(z)}(\eta) \, \frac{L_\eta(\|z\|_\infty\cdot \underline 1)\,L_\eta(\xi)}{|\eta|!\,L_\eta(\underline 1)} \, =\, f_{\nu(z)}(-\|z\|_\infty\cdot \underline 1) <\infty, $$
 which justifies the above limit. We now consider the  measures 
 $$ dm_\nu(x) := \frac{\Delta\bigl(\frac{\nu}{\xi}\bigr)^{\nu+\mu_0+1}}{\Gamma_n(\nu+\mu_0+1)} \cdot 1_{\mathbb R_+^n}(x) E\bigl(-\tfrac{\nu}{\xi},x\bigr) \Delta(x)^\nu \omega(x)dx \, \in M_b^+(\mathbb R^n).$$ 
 Due to Theorem \ref{master}, $m_\nu$ is actually a probability measure on $\mathbb R^n.$ 
 Formula \eqref{f_limit}, considered for arguments $z\in i\mathbb R^n$, 
 shows that the Dunkl transforms satisfy
 $$ \widehat{m_\nu} \to \widehat{\delta_{\xi}} \quad \text{ pointwise on } 
 \mathbb R^n,$$
 where $\delta_{\xi}$ denotes the point measure in $\xi$. 
Levy's continuity theorem (Lemma \ref{Levy}) now implies that $\,m_\nu \to \delta_{\xi}$ weakly. 
Thanks to the Portemanteau theorem (\cite{Kle14}) we even get
$$ \lim_{\nu\to\infty} \int_{\mathbb R^n} g\, dm_{\nu} = \int_{\mathbb R^n} g \,d\delta_{\xi} \, \,=\, g\bigl(\xi\bigr)$$
for all measurable bounded functions $g: \mathbb R^n\to \mathbb C$ which are continuous at $\xi$. 
Now suppose $f: \mathbb R_+^n \to \mathbb C$ is measurable, bounded and continuous at $\xi.$ 
Extend $f$ by zero to $\mathbb R^n.$ Then
$$\,\frac{\Delta\bigl(\frac{\nu}{\xi}\bigr)^{\nu+\mu_0+1}}{\Gamma_n(\nu+\mu_0+1)} \int_{\mathbb R_+^n} f(x) E\bigl(-\tfrac{\nu}{\xi} ,x\bigr) \Delta(x)^\nu \omega(x)dx \, =\, \int_{\mathbb R^n} f dm_\nu \, \to f(\xi).$$
But in view of to Lemma \ref{Laplaceconv} the integral on the left-hand side can be written as
$$ \mathcal L(\Delta^\nu f)\bigl(\tfrac{\nu}{\xi}\bigr) = \,\bigl(\Delta(-T)\bigr)^\nu(\mathcal Lf)\bigl(\tfrac{\nu}{\xi}\bigr),$$
which finishes the proof.
\end{proof}


\begin{thebibliography}{999}
\bibitem[ABHN01]{ABHN01} W. Arendt, C. Batty, M. Hieber, F. Neubrander, 
Vector-valued Laplace transforms and Cauchy problems.  Birkh\"auser Verlag, Basel, 2001.  
\bibitem[BF97]{BF97} T.H. Baker, P.J. Forrester, The Calogero-Sutherland model and generalized classical polynomials, \emph{Commun. Math. Phys. } 188 (1997), 175--216. 
\bibitem[BF98]{BF98} T.H. Baker, P.J. Forrester, Non-symmetric Jack 
polynomials and integral kernels. \emph{Duke Math. J.} 95 (1998), 1--50.
\bibitem[C63]{Con63} A.G. Constantine, Some non-central distribution problems in multivariate analysis, \emph{Ann. Math. Statist.} 34 (1963), 1270--1285. 
\bibitem[D89]{Dun89} C.F. Dunkl, Differential-difference operators
associated to reflection groups, \emph{Trans. Amer. Math. Soc.} 311 (1989), 167--183.
\bibitem[D91]{Dun91} C.F. Dunkl, Integral kernels with reflection group invariance, \emph{Canad. J. Math.} 43 (1991), 1213--1227. 
\bibitem[DX14]{DX14} C. Dunkl, Y. Xu, Orthogonal polynomials of Several Variables. Cambridge Univ. Press, 2nd edition, 2014.
 \bibitem[FG90]{FG90} J. Faraut, S. Gindikin, Deux formules d'inversion pour la transformation de Laplace sur un c\^one sym\'etrique.  \emph{C. R. Acad. Sci. Paris S\'er. I Math.} 310, 5--8  (1990). 
  \bibitem[FK94]{FK94} J. Faraut, A. Kor\'anyi, Analysis on Symmetric Cones. 
 Oxford Science Publications, Clarendon press, Oxford 1994.
 \bibitem[10]{For10} P.J. Forrester, Log-Gases and Random Matrices. London Mathematical Society Monographs Series, 34. Princeton University Press, Princeton, NJ, 2010.
 \bibitem[GR89]{GR89} K. Gross, D. Richards, Special functions of matrix argument. I: Algebraic induction, zonal polynomials, and hypergeometric functions. \emph{Trans. Amer Math. Soc.} 301 (1987), 781--811. 
\bibitem[HO21]{HO21} G. Heckman, E. Opdam,  Jacobi polynomials and hypergeometric functions associated with root systems. In: Encyclopedia of Special Functions, Part II: Multivariable Special Functions, eds. T.H. Koornwinder, J.V. Stokman, Cambridge University Press, Cambridge, 2021. 
\bibitem[He55]{Her55} C.S. Herz, Bessel functions of matrix argument. \emph{Ann. Math. } 61 (1955), 474--523. 
\bibitem[Ho20]{Hop20} F. Hoppe, Ein Inversionssatz f\"ur die Laplace-Transformation im Dunkl-Setting. Master Thesis, Paderborn University, 2020. 
\bibitem[dJ93]{dJ93} M. de Jeu, The Dunkl transform. \emph{Invent. Math.} 113, 147--162 (1993). 
\bibitem[Kad97]{Kad97}  K.W.J. Kadell, The Selberg-Jack symmetric functions, \emph{Adv. Math.} 130 (1997), 33--102
\bibitem[Kan93]{Kan93} J. Kaneko, Selberg integrals and hypergeometric functions
associated with Jack polynomials. 
\emph{SIAM J. Math. Anal.} 24 (1993),  1086--1100.
\bibitem[Kle14]{Kle14} A. Klenke, Probability. Springer-Verlag, London, 2nd ed. 2014. 
\bibitem[KS97]{KS97} F. Knop, S. Sahi, A recursion and a combinatorial formula for Jack polynomials. \emph{Invent. Math.} 128 (1997), 9--22.
\bibitem[KO08]{KO08} B. Kr\"otz, E. Opdam, Analysis on the crown domain. \emph{Geom. Funct. Anal. 18} (2008), 1326--1421.
\bibitem[M87]{Mac87} I.G. Macdonald, Commuting differential operators and zonal spherical functions. In: Algebraic groups (Utrecht 1986), eds. A.M. Cohen et al, \emph{Lecture Notes in Mathematics} 1271, Springer-Verlag, Berlin, 1987. 
\bibitem[M13]{Mac13} I.G. Macdonald, Hypergeometric functions I. arXiv: 1309.4568v1 (math.CA).
\bibitem[Mu82]{Mui82} R.J. Muirhead, Aspects of multivariate statistical theory. John Wiley \& Sons, Inc., New York, 1982. 
\bibitem[N10]{NIST} NIST Handbook of Mathematical Functions. Eds. F. Olver, D. Lozier, R. Boisvert and C. Clark. Cambridge University Press, Cambridge, 2010.
\bibitem[O95]{Opd95} E.M. Opdam, Harmonic analysis for certain representations of graded Hecke algebras. \emph{ Acta Math.}
 175, (1995), 75--112.
 \bibitem[R98]{R98} M. R\"osler,  Generalized Hermite polynomials and the heat equation for Dunkl operators. \emph{Commun. Math. Phys.} 192 (1998), 519--541. 
 \bibitem[R03]{R03} M. R\"osler, Dunkl operators: Theory and applications. In: E. Koelink, W. van Assche (Eds.), Lecture Notes in Math. 1817, Springer-Verlag, 2003, pp. 93--136.
\bibitem[R20]{R20} M. R\"osler, Riesz distributions and the Laplace transform in the Dunkl setting of type A. \emph{J. Funct. Anal.} 278 (2020), no 12, 108506, 29 pp. 
 \bibitem[RKV13]{RKV13} M. R\"osler, T. Koornwinder, M. Voit, Limit transition between hypergeometric functions of type BC and type A. \emph{Compos. Math.} 149 (2013), 1381--1400. 
 \bibitem[RV98]{RV98} M. R\"osler, M. Voit, Markov processes related with Dunkl operators. \emph{Adv. Appl. Math.} 21 (1998), 575--643. 
    \bibitem[31]{S98} S. Sahi, The binomial formula for nonsymmetric Macdonald polynomials. \emph{Duke Math. J.} 94 (1998),465-477.
 \bibitem[SZ07]{SZ07} S. Sahi, G. Zhang, Biorthogonal expansion of non-symmetric Jack functions.
SIGMA Symmetry Integrability Geom. Methods Appl. 3 (2007), Paper 106, 9 pp.
\bibitem[Sch08]{Sch08} B. Schapira, Contributions to the hypergeometric function theory of Heckman and Opdam: sharp estimates, Schwartz space, heat kernel, \emph{Geom. Funct. Anal.} 18 (2008), 222--250. 
    \bibitem[St89]{Sta89} R.P. Stanley, Some combinatorial properties 
of Jack symmetric functions. \emph{Adv. Math.} 77 (1989), 76--115.
\bibitem[T76]{Tit76} E.C. Titchmarsh, The theory of functions. Oxford Univ. Press, 1976.
 
\end{thebibliography}
\end{document}